\documentclass{amsart}
\usepackage{hyperref}
\usepackage{amsmath,amsthm,amssymb,comment,enumerate}
\usepackage{graphicx}
\usepackage{color}

\usepackage{subcaption}
\usepackage{float}
\usepackage{comment}

\usepackage{tikz}
\usetikzlibrary{backgrounds}

\usepackage{todonotes}

\theoremstyle{plain}
\newtheorem{thm}{Theorem}[section]
\newtheorem{cor}[thm]{Corollary}
\newtheorem{lem}[thm]{Lemma}
\newtheorem{prop}[thm]{Proposition}

\theoremstyle{definition}
\newtheorem{defn}[thm]{Definition}

\theoremstyle{remark}
\newtheorem{rem}[thm]{Remark}

\numberwithin{equation}{section}

\begin{document}

\title[Quasi-periodic BVPs for the Helmholtz equation]{The Functional Analytic Approach for Quasi-periodic Boundary Value Problems for the Helmholtz equation}%
\author{Roberto Bramati}%
\author{Matteo Dalla Riva}%
\author{Paolo Luzzini}%
\author{Paolo Musolino}%

\address{Roberto Bramati,  Department of Mathematics: Analysis, Logic and Discrete Mathematics Ghent University, Belgium {\tt roberto.bramati@ugent.be}}
\address{Paolo Luzzini, Dipartimento di  Matematica `Tullio Levi-Civita', Universit\`a degli Studi di Padova, Italy {\tt pluzzini@math.unipd.it}}%
\address{Paolo Musolino, Dipartimento di Scienze Molecolari e Nanosistemi, Universit\`a Ca’ Foscari Venezia, via Torino 155,
30172 Venezia Mestre, Italy {\tt paolo.musolino@unive.it }}%
\address{Matteo Dalla Riva, Dipartimento di Ingegneria, Universit\`a degli Studi di Palermo, Viale delle Scienze, Ed. 8, 90128 Palermo, Italy {\tt matteo.dallariva@unipa.it}.}


\maketitle

\date{\today}%
\begin{abstract} 
We lay down the preliminary work to apply  the Functional Analytic Approach to quasi-periodic boundary value problems for the Helmholtz equation. This consists in introducing a quasi-periodic fundamental solution and the related layer potentials,  showing how they are used to construct the solutions of quasi-periodic boundary value problems, and how they behave when we perform a singular perturbation of the domain. To show an application, we study a nonlinear quasi-periodic Robin problem in a domain with a set of holes that shrink to points.
\end{abstract}

\bigskip
\noindent



\setcounter{page}{1}
\tableofcontents

\section{Introduction}

Boundary value problems in singularly perturbed domains are widely studied in mathematics for their importance. A typical example consists of a boundary value problem on a smooth domain  with a hole of size $\epsilon>0$. When $\epsilon$ tends to $0$  the hole shrinks down to a point and a singularity appears on the boundary. Problems of this kind are relevant for the applications that they find in continuum mechanics,  applied sciences, and engineering. For instance, they play a central role in the analysis of  shape and topological optimization problems and of inverse problems related to nondestructive testing techniques. For a comprehensive description of these applications  we refer to the monographs of  Novotny and Soko\l owski \cite{NoSo13} and of Ammari and Kang \cite{AmKa07}.

Probably, the most common approach to  singular domain perturbation problems is that of Asymptotic Analysis, whose goal is to provide approximations of the solutions that are asymptotically correct as the perturbation parameter  $\epsilon$ tends to $0$. There are different ways to arrive to this kind of results. For example, one may resort to the Matched Asymtotic Expansion Method as in the book of Il'in \cite{Il92}, or  use the Multi-Scale   Expansion Method, which is also called Compound Expansion Method in the two-volume book of  Maz'ya,  Nazarov and Plamenewskii \cite{MaNaPl00a,MaNaPl00b}. Other techniques or variations of these two  can be found, for example, in  Kozlov,  Maz'ya and  Movchan \cite{KoMaMo99} (where the authors consider domains  depending on a small parameter $\epsilon$ in such a way that the limit regions consist of subsets of different space dimensions),  Bonnaillie-No\"el, Dambrine, Tordeux, and Vial \cite{BoDaToVi09} and Bonnaillie-No\"el, Dambrine, and  Lacave  \cite{BoDaLa16} (for problems in domains with two holes that collide to one another), and  in the  book of Maz'ya, Movchan, and Nieves \cite{MaMoNi13} (for the analysis of problems in a domain containing  ``clouds'' of small holes).  We also mention Dauge, Tordeux,  and  Vial~\cite{DaToVi10}, where the authors  compare the Multi-Scale and the Matched Asymptotic Expansion Methods in a corner perturbation problem.

Starting twenty years ago, a different approach was employed by Lanza de Cristoforis and his collaborators, the so-called Functional  Analytic Approach (see the seminal papers \cite{La02,La08}). In some cases, this approach is complementary 
 to the expansion methods of Asymptotic Analysis: Whereas the expansion methods produce asymptotic approximations of the solutions, the Functional  Analytic Approach aims at representing the exact solutions as real analytic maps and known functions of the perturbation parameters. Representation formulas of this type can then be used to expand the solutions as converging power series and compute explicitly (and in a constructive way) the coefficients. Moreover,   the Functional Analytic Approach has revealed to be extremely powerful  when dealing with nonlinear problems, which are not easily handled with the tools of Asymptotic Analysis. For a detailed description of the method and a careful comparison of its results with those of Asymptotic Analysis, we refer to the recent monograph \cite{DaLaMu21}.

The Functional Analytic Approach was successfully applied to a number of geometric situations and differential operators, including  linear and nonlinear problems for the Laplace equation in a domain with an interior hole \cite{La08,La10}, problems with two close holes \cite{DaMu16,DaMu17} or holes close to the boundary \cite{BoEtAl18,BoEtAl20}, perturbations near the vertex of a sector \cite{CoDaDaMu17}, elliptic systems \cite{Da13,DaLa11} and the Helmholtz equation \cite{AkLa22, AkLa22bis}. Periodic problems were studied with the purpose of analyzing the effective properties of composite materials \cite{DaLuMuPu22, LuMu20} and dilute composites \cite{DaMuPu19}.

The present  paper is the first application of the Functional Analytic Approach to quasi-periodic problems. More specifically, we will deal with quasi-periodic problems for the Helmholtz equation. These problems find applications in spectral theory, because periodic eigenvalue problems can be transformed into families of problems for quasi-periodic functions  (see, e.g., Nazarov, Ruotsalainen, and Taskinen \cite{NaRuTa12}, Ferraresso and Taskinen \cite{FeTa15}). Also, quasi-periodic problems for the Helmholtz equation arise for example in scattering theory and in the study of metamaterials and phononic crystals (see Ammari and collaborators \cite{AmFiGoLeZh17, AmFiLeYu17, AmHiYu22, AmKaSoZr06}, Aylwin, Jerez-Hanckes, and Pinto \cite{AyJePi20}, Bruno and Fernandez-Lado \cite{BrFe20}, Bruno and Reitich \cite{BrRe92}, Bruno, Shipman, Turc, and Venakides \cite{BrShTuVe17}, Pérez-Arancibia, Shipman, Turc, and Venakides \cite{PeShTuVe19}).

Since the standard application of the Functional Analytic Approach is based on potential theory, our first step is to recall the construction of a quasi-periodic fundamental solution for the Helmholtz equation and of the related layer potentials (see Sections \ref{s:fund} and \ref{s:pot}). Then we use these layer potentials to construct the solutions of quasi-periodic boundary value problems (see Section \ref{s:qperbvp}). Next, we turn to singular perturbation problems and we study the behavior of quasi-periodic  layer potentials supported on the boundary of singularly perturbed periodic domains (this is done in Section \ref{s:singpot}). Finally, we show an application of the results so far collected. Out of the many problems we may choose to study, we opt for a nonlinear Robin problem on an infinite domain with a periodic set of holes of size $\epsilon$ (see Section \ref{s:singbvp}). This problem allows to illustrate some of the features of the Functional Analytic Approach, but there is no other special reason for our choice. Scientists interested in more specific applications may try to follow  our blueprint and modify the computations. Possible variants may include different periodic geometries, different boundary conditions, and different operators.

\section{Quasi-periodic fundamental solutions for the Helmholtz equation}\label{s:fund} 

In this section, we introduce certain fundamental solutions for the Helmholtz equation. That is, for the equation $\Delta u+k^2 u=0$ with $k\in \mathbb{C}$ and $\Delta:=\sum_{j=1}^n\partial^2_{x_j}$.
The complex number $k^2$ is often referred to as the wave number. 
 Our final goal is to study quasi-periodic  problems   using a quasi-periodic version of the potential theory for the Helmholtz equation. So, after presenting a family of fundamental solutions for the standard (not periodic) equation, we will define its quasi-periodic counterpart and the related quasi-periodic layer potentials. 

\subsection{A family of fundamental solutions for the Helmholtz equation}\label{s:holfundsol}


We present the family of fundamental solutions $\{S_n(\cdot,k)\}_{k \in \mathbb{C}}$ that was introduced by Lanza de Cristoforis and Rossi in \cite{LaRo08}. One feature of this family is that the functions $S_n(\cdot,k)$ depend holomorphically on the parameter $k\in \mathbb{C}$. This property will come handy when studying  the effect of singular domain perturbations on the quasi-periodic layer potentials (see Section \ref{s:singpot}).

Since for $k=0$ the operator $\Delta+k^2$ is just $\Delta$, we start with the classical fundamental solution of the Laplace operator. That is, the function $S_{n}$  from ${\mathbb{R}}^{n}\setminus\{0\}$ to ${\mathbb{R}}$ defined by
\[
S_{n}(x):=
\left\{
\begin{array}{lll}
\frac{1}{s_{2}}\log |x| \qquad &   \forall x\in
{\mathbb{R}}^{{2}}\setminus\{0\},\quad & {\mathrm{if}}\ n=2\,,
\\
\frac{1}{(2-n)s_{n}}|x|^{2-n}\qquad &   \forall x\in
{\mathbb{R}}^{n}\setminus\{0\},\quad & {\mathrm{if}}\ n\geq3\,,
\end{array}
\right.
\]
where $s_n$ denotes the $(n-1)$-dimensional measure of the unit sphere in $\mathbb{R}^n$.

Then, we denote by $\gamma$, $\Gamma$, $J_\nu$, $N_{\nu}$ the Euler constant, the Euler Gamma function, the Bessel function of order $\nu \in \mathbb{R}$, and the Neumann function of order $\nu \in \mathbb{R}$, respectively (as in Schwartz \cite[Ch. VIII, IX]{Sc65}). We have the following technical lemma (see Lanza de Cristoforis and Rossi \cite[Lemma 3.1]{LaRo08}).

\begin{lem}\label{lmm:hlmtec}
Let $n \in \mathbb{N} \setminus \{0,1\}$. Then the following statements hold.
\begin{itemize}
\item[i)] If $n$ is even, then the map from $]0,+\infty[$ to $\mathbb{R}$ that takes $t$ to
\[
t^{\frac{n-2}{2}}\Bigl\{N_{\frac{n-2}{2}}(t)-\frac{2}{\pi}(\log(t/2)+\gamma)J_{\frac{n-2}{2}}(t)\Bigr\} \qquad \forall t \in \mathopen]0,+\infty[,
\]
admits a unique holomorphic extension $\tilde{N}_{\frac{n-2}{2}}(\cdot)$ from $\mathbb{C}$ to $\mathbb{C}$, and
 $\tilde{N}_{\frac{n-2}{2}}(0)=-\pi^{-1}2^{\frac{n-2}{2}}(\frac{n-4}{2})!$ for {$n\geq 4$}, and $\lim_{t\to 0}\tilde{N}_{\frac{n-2}{2}}(t) t^{-2}=\frac{1}{2\pi}$ for $n=2$.
\item[ii)] If $\nu \in \mathbb{R}$, then the map from $]0,+\infty[$ to $\mathbb{R}$ that takes $t$ to $t^{-\nu}J_{\nu}(t)$ admits a unique holomorphic extension $\tilde{J}_{\nu}$ from $\mathbb{C}$ to $\mathbb{C}$.
\item[iii)]If $n$ is even, then we have $\tilde{J}_{\frac{n-2}{2}}(0)=2^{-\frac{n-2}{2}}/(\frac{n-2}{2})!$. If $n$ is odd, then we have $\tilde{J}_{-\frac{n-2}{2}}(0)=(-1)^{\frac{n-3}{2}}\frac{2^{\frac{n-2}{2}}}{\pi}(\frac{n-2}{2})^{-1}\Gamma(n/2)$. 
\end{itemize}
\end{lem}

Before introducing the fundamental solution $S_n(\cdot,k)$ of $\Delta + k^2$, we need to introduce some further notation (cf. Lanza de Cristoforis and Rossi \cite[Def.~3.2]{LaRo08}).
\begin{defn}\label{def:hlmups}
Let $n \in \mathbb{N} \setminus \{0,1\}$.
\begin{itemize}
\item[i)]If $n$ is even, then we set
\begin{align*}
& \mathcal{J}_n(z):= (2\pi)^{-n/2}\tilde{J}_{\frac{n-2}{2}}(z),\\ \notag
& \mathcal{N}_n(z):= 2^{-(n/2)-1}\pi^{-(n/2)+1}\tilde{N}_{\frac{n-2}{2}}(z), \notag
\end{align*}
for all $z \in \mathbb{C}$.
\item[ii)] If $n$ is odd, then we set 
\begin{align*}
& \mathcal{J}_n(z):= 0,\\ \notag
& \mathcal{N}_n(z):= (-1)^{\frac{n-1}{2}}2^{-(n/2)-1}\pi^{-(n/2)+1}\tilde{J}_{-\frac{n-2}{2}}(z), \notag
\end{align*}
for all $z \in\mathbb{C}$.
\item[iii)]We set
\[
\Upsilon_n(r,k):= k^{n-2}\mathcal{J}_n(rk)\log r+\frac{\mathcal{N}_n(rk)}{r^{n-2}},
\]
for all $(r,k) \in \mathopen]0,+\infty\mathclose[\times \mathbb{C}$.  
\end{itemize}
\end{defn}
Here, we agree that $0^0=1$. Then we have the following (see Lanza de Cristoforis and Rossi \cite[Prop.~3.3]{LaRo08}).
\begin{prop}\label{prop:hlmfs}
 Let $n \in \mathbb{N} \setminus \{0,1\}$. Then the following statements hold.
\begin{itemize}
\item[i)] $\mathcal{J}_2(0)=\frac{1}{2\pi}$ and if $n\ge 4$ is even then $\mathcal{J}_n(0)=2^{1-n}\pi^{-n/2}/(\frac{n-2}{2})!$.\\ $\mathcal{N}_2(0)=0$ and $\mathcal{N}_n(0)=(2-n)^{-1}s_n^{-1}$ for all $n \geq 3$.
\item[ii)] $\mathcal{J}_n$ and $\mathcal{N}_n$ are entire holomorphic functions. The function $\Upsilon_n$ is real analytic on $\mathopen]0,+\infty\mathclose[\times \mathbb{C}$.
\item[iii)] Let $k \in \mathbb{C}$. The function  ${S_n(\cdot,k)}$ from $\mathbb{R}^n \setminus \{0\}$ to $\mathbb{C}$ defined by 
\[
{S_n(x,k):=\Upsilon_n(|x|,k)}
\]
 for all $x \in \mathbb{R}^n \setminus \{0\}$ is a fundamental solution of $\Delta+k^2$. In particular, $S_n(\cdot,0)$ is the usual fundamental solution  $S_n$ of the Laplace operator.
\end{itemize}
\end{prop}

One may observe that $S_n(\cdot,k)$ does not coincide with the fundamental solution of $\Delta +k^2$ that is commonly used in scattering theory (cf. e.g. Colton and Kress \cite{CoKr13}). The advantage of working with  $S_n(\cdot,k)$ will be clarified in Section \ref{s:singpot}, where we will exploit {its} holomorphic dependence on $k\in \mathbb{C}$. 

\subsection{A quasi-periodic fundamental solution for the Helmholtz equation}

We now show how we can use Fourier analysis to define a quasi-periodic fundamental solution for the Helmholtz equation.  For this construction we refer, for example, to Ammari, Kang and Lee \cite[p. 123]{AmKaLe09}, Ammari, Kang, Soussi and Zribi \cite{AmKaSoZr06}, Dienstfrey, Hang and Huang \cite{DiHaHu01}, Linton \cite{Li98},  Poulton, Botten, McPhedran and Movchan \cite{PoBoMcMo99}.

Let $n \in \mathbb{N}$, $n \geq 2$  represent the dimension of the space.
 We take 
\[
(q_{11},\ldots,q_{nn}) \in \mathopen]0,+\infty[^n
\] 
and we define a periodicity cell $Q \subseteq \mathbb{R}^n$ and a matrix $q \in {\mathbb{D}}_{n}^{+}({\mathbb{R}})$ by
\[
Q := \prod_{j=1}^n \mathopen]0,q_{jj}[, \quad
q := 
 \begin{pmatrix}
  q_{11} & 0 & \cdots & 0 \\
  0 & q_{22} & \cdots & 0 \\
  \vdots  & \vdots  & \ddots & \vdots  \\
  0 & 0 & \cdots & q_{nn} 
 \end{pmatrix},
\]
where {${\mathbb{D}}^+_{n}({\mathbb{R}})$} is the space of $n\times n$ diagonal matrices with {positive} real entries {on the diagonal}. 
We denote by $|Q|_n$  the $n$-dimensional measure of the cell $Q$, by $\nu_Q$ the outward unit normal to $\partial Q$, where it exists, and by $q^{-1}$ the inverse matrix of $q$.

Let $\eta \in \mathbb{R}^n$. A function $f: \mathbb{R}^n \to \mathbb{C}$ is said to be $\eta$-quasi-periodic with respect to $Q$, or simply $(Q,\eta)$-quasi-periodic, if $f(x) e^{-i\eta \cdot x}$ is periodic with respect to Q,  that is if 
\[
f(x+qe_h)e^{-i\eta \cdot (x+qe_h)} = f(x) e^{-i\eta \cdot x} \qquad \forall x\in \mathbb{R}^n, \,\forall h \in \{1,\ldots,n\},  
\]
where $e_1,\ldots,e_n$ is the standard basis of $\mathbb{R}^n$.  Clearly, if $j \in\{1,\dots,n\}$ a function $f$ is $(Q,\eta)$-quasi-periodic if and only if it is $(\eta+2\pi q_{jj}^{-1}e_j,Q)$-quasi-periodic. As a consequence, it would suffice to consider the case $\eta \in {\prod_{j=1}^n[0,2\pi q_{jj}^{-1}[}$. However, for the sake of brevity we will write $\eta \in \mathbb{R}^n$.

Let $k \in \mathbb{C}$. We now introduce a $(Q,\eta)$-quasi-periodic distribution that will play the role of the fundamental solution of the Helmholtz operator $\Delta+k^2$.
We set   
\[
Z_{q,\eta}(k) := \left\{z \in \mathbb{Z}^n : k^2 = |2\pi q^{-1}z +\eta|^2\right\}.
\]
One can verify that the set $Z_{q,\eta}(k)$ is finite. Let  $G^{k}_{q,\eta}$ be defined by the generalized series
\begin{equation*}
G^{k}_{q,\eta} := \sum_{z \in \mathbb{Z}^n \setminus Z_{q,\eta}(k)  } \frac{1}{|Q|_n(k^2-|2\pi q^{-1}z +\eta|^2)}E_{2\pi q^{-1}z +\eta}.
\end{equation*}
We have denoted by $E_{2\pi q^{-1}z +\eta}$  the distribution associated with the function $e^{ix\cdot (2\pi q^{-1}z +\eta)}$. A standard argument shows that the above generalized series defines a tempered distribution in $\mathcal{S}'(\mathbb{R}^n)$ (see e.g. \cite[Thm. 12.2 p. 486]{DaLaMu21}). In the next proposition we show that the distribution $G^{k}_{q,\eta}$ enjoys a property that allows it to be exploited as an analog of the fundamental solution in the quasi-periodic setting.
\begin{prop}\label{qpfund}
Let $q \in \mathbb{D}^+_n(\mathbb{R})$, $\eta \in \mathbb{R}^n$, and $k \in \mathbb{C}$. Then
$G^{k}_{q,\eta}$  is $(Q,\eta)$-quasi-periodic  in the sense of distributions and 
\begin{equation}\label{qpfund1}
(\Delta+k^2)G^{k}_{q,\eta} = \sum_{z \in \mathbb{Z}^n}\delta_{qz}e^{iqz\cdot \eta} - \sum_{z \in Z_{q,\eta}(k)}\frac{1}{|Q|_n}E_{2\pi q^{-1}z +\eta}
\end{equation}
in the sense of distributions. Here above, the symbol $\delta_{qz}$ denotes the Dirac delta distribution with mass concentrated at $qz$.
\end{prop}
\begin{proof}
Since $G^{k}_{q,\eta}$  is a generalized sum of distributions that are  $(Q,\eta)$-quasi-periodic, then it is a $(Q,\eta)$-quasi-periodic distribution.  Next we pass to prove equality \eqref{qpfund1}. A direct computation of the distributional derivatives of $G^{k}_{q,\eta}$  shows that 
\[
\partial_{x_j} G^{k}_{q,\eta} = \sum_{z \in \mathbb{Z}^n \setminus Z_{q,\eta}(k)  } \frac{i(2\pi q^{-1}_{jj}z_j+\eta_j)}{|Q|_n(k^2-|2\pi q^{-1}z +\eta|^2)}E_{2\pi q^{-1}z +\eta},
\]
\[
\partial^2_{x_j} G^{k}_{q,\eta} = -\sum_{z \in \mathbb{Z}^n \setminus Z_{q,\eta}(k)  } \frac{(2\pi q^{-1}_{jj}z_j+\eta_j)^2}{|Q|_n(k^2-|2\pi q^{-1}z +\eta|^2)}E_{2\pi q^{-1}z +\eta},
\]
and accordingly
\begin{equation}\label{qpfund2}
\Delta G^{k}_{q,\eta}  = -\sum_{z \in \mathbb{Z}^n \setminus Z_{q,\eta}(k)  } \frac{|2\pi q^{-1}z+\eta|^2}{|Q|_n(k^2-|2\pi q^{-1}z +\eta|^2)}E_{2\pi q^{-1}z +\eta}.
\end{equation}
 Then, by equality \eqref{qpfund2} and by the Poisson summation formula (see e.g. Folland \cite[ p. 254]{Fo99}) one has
\begin{align*}
(\Delta+k^2) G^{k}_{q,\eta} &= \sum_{z \in \mathbb{Z}^n \setminus Z_{q,\eta}(k)  } \frac{k^2-|2\pi q^{-1}z+\eta|^2}{|Q|_n(k^2-|2\pi q^{-1}z +\eta|^2)}E_{2\pi q^{-1}z +\eta} \\
&= \sum_{z \in \mathbb{Z}^n \setminus Z_{q,\eta}(k)  } \frac{1}{|Q|_n }E_{2\pi q^{-1}z +\eta} \\
&= \sum_{z \in \mathbb{Z}^n  } \frac{1}{|Q|_n }E_{2\pi q^{-1}z +\eta} -\sum_{z \in Z_{q,\eta}(k)  } \frac{1}{|Q|_n }E_{2\pi q^{-1}z +\eta} \\
&= \sum_{z \in \mathbb{Z}^n}\delta_{qz}e^{iqz\cdot \eta}  -\sum_{z \in Z_{q,\eta}(k)  } \frac{1}{|Q|_n }E_{2\pi q^{-1}z +\eta}.
\end{align*}
\end{proof}
In order to understand some regularity properties of $G^{k}_{q,\eta}$, we compare it with a general fundamental solution  $\mathbf{G}$  of the Helmholtz operator $\Delta+k^2$ in $\mathbb{R}^n$, like for example could be the function $S_n(\cdot,k)$ 
of Proposition \ref{prop:hlmfs}. 
We note that the idea of comparing an analog of a fundamental {solution in} a periodic setting with a classical fundamental solution has 
been already used in \cite{Mu12} for the periodic Laplace equation, in \cite{LaMu18} for the periodic Helmholtz equation, and in \cite{Lu20} for the periodic heat equation.
\begin{prop}\label{qnSR}
Let $q \in \mathbb{D}^+_n(\mathbb{R})$, $\eta \in \mathbb{R}^n$, and $k \in \mathbb{C}$.
Let $\mathbf{G}$ be a fundamental solution of  the Helmholtz operator $\Delta+k^2$ in $\mathbb{R}^n$.
 Then
the following statements hold.
\begin{itemize}
\item[i)] The distribution $R_{\mathbf{G}} := G^{k}_{q,\eta}-\mathbf{G}$ comes from a real analytic function in $(\mathbb{R}^n\setminus q\mathbb{Z}^n) \cup \{0\}$.
\item[ii)] $G^{k}_{q,\eta}$  comes from a real analytic function   in $\mathbb{R}^n\setminus q\mathbb{Z}^n$. 
\item[iii)] $G^{k}_{q,\eta}$ is in $L^1_{\mathrm{loc}}(\mathbb{R}^n)$.
\end{itemize}
\end{prop}
\begin{proof}
We consider statement i).
By Proposition \ref{qpfund} and since $\mathbf{G}$ is a fundamental solution for the Helmholtz operator, one has
 \[
(\Delta+k^2)R_{\mathbf{G}} = \sum_{z \in \mathbb{Z}^n\setminus\{0\}}\delta_{qz}e^{iqz\cdot \eta} - \sum_{z \in Z_{q,\eta}(k)}\frac{1}{|Q|_n}E_{2\pi q^{-1}z +\eta}.
\]
Since $\sum_{z \in Z_{q,\eta}(k)}\frac{1}{|Q|_n}E_{2\pi q^{-1}z +\eta}$ is  real analytic and the distribution
$\sum_{z \in \mathbb{Z}^n\setminus\{0\}}\delta_{qz}e^{iqz\cdot \eta}$ vanishes in $(\mathbb{R}^n\setminus q\mathbb{Z}^n) \cup \{0\}$, then by classical elliptic regularity theory $R_{\mathbf{G}}$ is real analytic in $(\mathbb{R}^n\setminus q\mathbb{Z}^n) \cup \{0\}$.

Statement ii) follows in a similar way  by Proposition \ref{qpfund} and standard elliptic regularity theory. 

Finally, statement iii)  follows by the real analyticity of  $R_{\mathbf{G}}$  in $(\mathbb{R}^n\setminus q\mathbb{Z}^n) \cup \{0\}$, by the $(Q,\eta)$-quasi-periodicity of 
$G^{k}_{q,\eta}$ and by the local integrability of  $\mathbf{G}$ in $\mathbb{R}^n$ (see John \cite[Ch. III]{Jh55}).
\end{proof}

\section{Layer potentials}\label{s:pot} 

Let $\alpha \in \mathopen]0,1[$, $q \in \mathbb{D}^+_n(\mathbb{R})$, $\eta \in \mathbb{R}^n$, and $k \in \mathbb{C}$. 
{For the definition and properties of functions and sets of the Schauder class $C^{m,\alpha}$, $m \in \mathbb{N}$, we refer to 
Gilbarg and Trudinger \cite{GiTr83}.}
Let $\Omega_Q$ be a bounded open subset of $\mathbb{R}^n$ of class $C^{1,\alpha}$  such that $\overline{\Omega_Q} \subseteq Q$.  We define the following two periodic open sets:
\[
\mathbb{S}_q[\Omega_Q] := \bigcup_{z \in \mathbb{Z}^n}(qz+\Omega_Q ), 
\qquad \mathbb{S}^-_q[\Omega_Q] := \mathbb{R}^n \setminus \overline{\mathbb{S}_q[\Omega_Q]}.
\]
{As we have done for functions defined on $\mathbb{R}^n$, a function $f$ from $\overline{\mathbb{S}_q[\Omega_Q]}$ or $\overline{\mathbb{S}^-_q[\Omega_Q]}$ is said to be $(Q,\eta)$-quasi-periodic if 
\[
f(x+qe_h)e^{-i\eta \cdot (x+qe_h)} = f(x) e^{-i\eta \cdot x} 
\]
for all $x$ in the domain of $f$ and for all $h \in \{1,\ldots,n\}$.}  We now introduce layer potentials   where the role of the standard fundamental solution is taken by $G^{k}_{q,\eta}$.
We start with the double layer potential. Let $\mu \in C^0(\partial\Omega_Q)$. The $(Q,\eta)$-quasi-periodic double layer potential for the Helmholtz equation is 
\[
\mathcal{D}^{k}_{q,\eta}[\partial\Omega_Q,\mu](x) := -\int_{\partial \Omega_Q}\nu_{\Omega_Q}(y) \cdot\nabla G^{k}_{q,\eta}(x-y)\mu(y)\,d\sigma_y \qquad \forall x \in \mathbb{R}^n.
\]
 Moreover, we set
\[
\mathcal{K}^{k}_{q,\eta}[\partial\Omega_Q,\mu] := \mathcal{D}^{k}_{q,\eta}[\partial\Omega_Q,\mu]_{|\partial\Omega_Q} \qquad \mbox{ on } \partial\Omega_Q.
\]
In the next proposition we collect some properties of the $(Q,\eta)$-quasi-periodic double layer potential in Schauder spaces. One may observe that these properties are the $(Q,\eta)$-quasi-periodic counterpart of the 
analog properties exhibited by the standard double layer potential. 

Another natural setting for potential theory would be that of Sobolev spaces. We opt for Schauder spaces because there is some advantage when dealing with nonlinear problems, as we do in Section \ref{s:singbvp}. This is due to the fact that Schauder spaces are Banach algebras and Sobolev spaces are not.

\begin{prop}\label{prop:dlp}
Let $\alpha \in \mathopen]0,1[$, $q \in \mathbb{D}^+_n(\mathbb{R})$, $\eta \in \mathbb{R}^n$, and $k \in \mathbb{C}$. Let $\Omega_Q$ be a bounded open subset of $\mathbb{R}^n$ of class $C^{1,\alpha}$  such that $\overline{\Omega_Q} \subseteq Q$.  
Let $\mu \in C^{1,\alpha}(\partial \Omega_Q)$. Then the following statements hold.
\begin{itemize}
\item[i)]  $\mathcal{D}^{k}_{q,\eta}[\partial\Omega_Q,\mu]$ is of class $C^\infty(\mathbb{S}_q[\Omega_Q] \cup \mathbb{S}^-_q[\Omega_Q])$ and
\begin{align*}
(\Delta&+k^2)\mathcal{D}^{k}_{q,\eta}[\partial\Omega_Q,\mu](x) \\
&= \frac{1}{|Q|_n}\sum_{z \in Z_{q,\eta}(k)}e^{ix\cdot (2\pi  q^{-1}z+\eta)}i(2\pi  q^{-1}z+\eta)\cdot \int_{\partial \Omega_Q}\nu_{\Omega_Q}(y)e^{-iy\cdot (2\pi  q^{-1}z+\eta)} \mu(y)\,d\sigma_y\\
&\hspace{9cm}\forall x \in \mathbb{S}_q[\Omega_Q] \cup \mathbb{S}^-_q[\Omega_Q].
\end{align*}
\item[ii)] $\mathcal{D}^{k}_{q,\eta}[\partial\Omega_Q,\mu]$ is $(Q,\eta)$-quasi-periodic.
\item[iii)] The restriction $\mathcal{D}^{k}_{q,\eta}[\partial\Omega_Q,\mu]_{|\mathbb{S}_q[\Omega_Q]}$ can be extended to a continuous function $\mathcal{D}^{k,+}_{q,\eta}[\partial\Omega_Q,\mu] \in C^{1,\alpha}(\overline{\mathbb{S}_q[\Omega_Q]})$ 
and the restriction $\mathcal{D}^{k}_{q,\eta}[\partial\Omega_Q,\mu]_{|\mathbb{S}^-_q[\Omega_Q]}$ can be extended to a continuous function $\mathcal{D}^{k,-}_{q,\eta}[\partial\Omega_Q,\mu] \in C^{1,\alpha}(\overline{\mathbb{S}^-_q[\Omega_Q]})$. Moreover
\begin{align} \label{jfd}
&\mathcal{D}^{k,\pm}_{q,\eta}[\partial\Omega_Q,\mu] = \pm\frac{1}{2}\mu+\mathcal{K}^{k}_{q,\eta}[\partial\Omega_Q,\mu]  \qquad \mbox{ on } \partial\Omega_Q,\\ \nonumber
&\nu_{\Omega_Q}\cdot \nabla \mathcal{D}^{k,+}_{q,\eta}[\partial\Omega_Q,\mu] -\nu_{\Omega_Q}\cdot \nabla \mathcal{D}^{k,-}_{q,\eta}[\partial\Omega_Q,\mu] =0  \qquad \mbox{ on } \partial\Omega_Q.
\end{align}
\item[iv)]  The map from $C^{1,\alpha}(\partial \Omega_Q)$ to $ C^{1,\alpha}(\overline{\mathbb{S}_q[\Omega_Q]})$ 
that takes $\mu$ to $\mathcal{D}^{k,+}_{q,\eta}[\partial\Omega_Q,\mu]$ and the map from $C^{1,\alpha}(\partial \Omega_Q)$ to $ C^{1,\alpha}(\overline{\mathbb{S}^-_q[\Omega_Q]})$ 
that takes $\mu$ to $\mathcal{D}^{k,-}_{q,\eta}[\partial\Omega_Q,\mu]$ are linear and continuous.
\item[v)] The map that takes $\mu \in C^{1,\alpha}(\partial\Omega_Q)$ to $\mathcal{K}^{k}_{q,\eta}[\partial\Omega_Q,\mu]$ is a compact operator from  $C^{1,\alpha}(\partial\Omega_Q)$ 
to itself.
\end{itemize}
\end{prop}
\begin{proof}
First, we note that 
 \[
 x-y \notin q\mathbb{Z}^n \qquad \forall\,(x,y) \in (\mathbb{R}^n \setminus \partial \mathbb{S}_q[\Omega_Q])\times \partial \Omega_Q.
 \]
 Indeed, if we assume for the sake of contradiction that  
 $(x,y) \in (\mathbb{R}^n \setminus \partial \mathbb{S}[\Omega_Q])\times \partial \Omega_Q$ and $x-y \in q\mathbb{Z}^n$, then we can deduce that
 $x \in \partial \Omega_Q + q\mathbb{Z}^n = \partial \mathbb{S}[\Omega_Q]$, contrary to our assumption on $x$. 
Then statement i) is a consequence of classical differentiation theorems for integrals depending on a parameter and of {Propositions \ref{qpfund} and \ref{qnSR}}.

Next we consider statement ii). Since $G^{k}_{q,\eta}$ is $(Q,\eta)$-quasi-periodic, it is easily seen that so it is $\nabla G^{k}_{q,\eta}$. Indeed for all $ x \in \mathbb{R}^n \setminus {q\mathbb{Z}^n}$ one has
\begin{align*}
e^{-i\eta\cdot x}\nabla G^{k}_{q,\eta}(x) &= e^{-i\eta\cdot x}\nabla (G^{k}_{q,\eta}(x) e^{-i\eta\cdot x}e^{i\eta\cdot x}) \\
&=
\nabla (G^{k}_{q,\eta}(x) e^{-i\eta\cdot x}) +i\eta G^{k}_{q,\eta}(x)e^{-i\eta\cdot x}
\end{align*}
and both the terms in the right hand side of the above formula are periodic with respect to $Q$. As a consequence,   $\mathcal{D}^{k}_{q,\eta}[\partial\Omega_Q,\mu]$ is $(Q,\eta)$-quasi-periodic.

Next we pass to prove statement   iii). We  apply  Proposition \ref{qnSR} where we chose the fundamental 
solution of the Helmholtz operator $\mathbf{G}$ to be the one denoted by $S_n(\cdot,k)$ and introduced in  Subsection \ref{s:holfundsol} (see
Lanza de Cristoforis and Rossi \cite[Prop. 3.3]{LaRo08}).
With this choice:
\begin{align}\label{dlp1}
\mathcal{D}^{k}_{q,\eta}[\partial\Omega_Q,\mu](x) = \mathcal{D}^{k}[\partial\Omega_Q,\mu](x) -\int_{\partial \Omega_Q}\nu_{\Omega_Q}(y) \cdot\nabla &R_{S_n(\cdot,k)}(x-y)\mu(y)\,d\sigma_y\\ \nonumber
& \qquad \forall x \in \mathbb{R}^n,
\end{align}
where $\mathcal{D}^{k}[\partial\Omega_Q,\mu]$ is the double layer potential constructed with the fundamental solution 
$S_n(\cdot,k)$  and $R_{S_n(\cdot,k)}$ is the map defined in Proposition \ref{qnSR} with the choice $\mathbf{G} =S_n(\cdot,k)$. 
As it is well known, the restriction $\mathcal{D}^{k}[\partial\Omega_Q,\mu]_{|\Omega_Q}$ can be extended to a continuous function $\mathcal{D}^{k,+}[\partial\Omega,\mu] \in C^{1,\alpha}(\overline{ \Omega_Q})$ 
and the restriction $\mathcal{D}^{k}[\partial\Omega_Q,\mu]_{|\Omega_Q^-}$ can be extended to a continuous function $\mathcal{D}^{k,-}[\partial\Omega_Q,\mu] \in C_{\mathrm{loc}}^{1,\alpha}(\overline{\Omega_Q^-})$  (see e.g. Lanza de Cristoforis and Rossi \cite[Thm 3.4]{LaRo08}). 
We now take $A$ to be a bounded open subset of $\mathbb{R}^n$ of class $C^\infty$ such that 
\[
\overline{Q} \subseteq A, \qquad \overline{A} \cap (qz+\overline{\Omega_Q}) =\emptyset \qquad \forall z \in \mathbb{Z}^n \setminus\{0\}.
\] 
We moreover set 
\[
B := A \setminus \overline{\Omega_Q}.
\]
We first note that if $x \in \overline{A}$ and $y \in \partial \Omega_Q$, then 
\[
x-y \notin q\mathbb{Z}^n\setminus\{0\}.
\]
Indeed if by contradiction  $x-y \in  q\mathbb{Z}^n\setminus\{0\}$, then $x \in \partial \Omega_Q +  (q\mathbb{Z}^n\setminus\{0\})$
and thus there exists $z \in  \mathbb{Z}^n\setminus\{0\}$ such that $\overline{A} \cap (qz+\partial \Omega_Q) \neq \emptyset $ which cannot be.
Thus, since by Proposition \ref{qnSR} i) the map $R_{S_n(\cdot,k)}$ is real analytic in $(\mathbb{R}^n\setminus q\mathbb{Z}^n) \cup \{0\}$, the second term in 
the right hand side of \eqref{dlp1} is of class $C^\infty(A)$.  Then 
\begin{align*}
\mathcal{D}^{k}_{q,\eta}[\partial\Omega_Q,\mu](x) = \mathcal{D}^{k,+}[\partial\Omega_Q,\mu](x) -\int_{\partial \Omega_Q}\nu_{\Omega_Q}(y) \cdot\nabla &R_{S_n(\cdot,k)}(x-y)\mu(y)\,{d\sigma_y}\\ \nonumber
& \qquad \forall x \in \Omega_Q,\\
\mathcal{D}^{k}_{q,\eta}[\partial\Omega_Q,\mu](x) = \mathcal{D}^{k,-}[\partial\Omega_Q,\mu](x) -\int_{\partial \Omega_Q}\nu_{\Omega_Q}(y) \cdot\nabla &R_{S_n(\cdot,k)}(x-y)\mu(y)\,{d\sigma_y}\\ \nonumber
& \qquad \forall x \in B.
\end{align*}
Since the right hand side of the above equations define respectively two functions in $C^{1,\alpha}(\overline{\Omega_Q})$ and 
$C^{1,\alpha}(\overline{B})$, it is readily seen that  $\mathcal{D}^{k}_{q,\eta}[\partial\Omega_Q,\mu]_{|\mathbb{S}_q[\Omega_Q]}$ can be extended to a continuous function $\mathcal{D}^{k,+}_{q,\eta}[\partial\Omega_Q,\mu] \in C^{1,\alpha}(\overline{\mathbb{S}_q[\Omega_Q]})$ 
and $\mathcal{D}^{k}_{q,\eta}[\partial\Omega_Q,\mu]_{|\mathbb{S}^-_q[\Omega_Q]}$ can be extended to a continuous function $\mathcal{D}^{k,-}_{q,\eta}[\partial\Omega_Q,\mu] \in C^{1,\alpha}(\overline{\mathbb{S}^-_q[\Omega_Q]})$. Note that a $(Q,\eta)$-quasi-periodic function is completely determined by its behavior on a single cell. The formulas in \eqref{jfd} follow by the corresponding  classical formulas for $\mathcal{D}^{k}[\partial\Omega_Q,\mu]$ (see 
 Lanza de Cristoforis and Rossi \cite[Thm 3.4]{LaRo08}).
 
 Statements iv) and v) similarly follow  from formula \eqref{dlp1}, from the mapping properties of $\mathcal{D}^{k}[\partial\Omega,\cdot]$  {for the first term in the right hand side of  \eqref{dlp1}}
 (see e.g. Lanza de Cristoforis and Rossi \cite{LaRo08}){,} and from the mapping properties of integral operators with a smooth kernel {for the second term in  \eqref{dlp1}} (see, e.g., \cite{LaMu13}). 
\end{proof}

Then we pass to the single layer potential.  Let $\mu \in C^0(\partial\Omega_Q)$. The $(Q,\eta)$-quasi-periodic single layer potential for the Helmholtz equation is 
\[
\mathcal{S}^{k}_{q,\eta}[\partial\Omega_Q,\mu](x) := \int_{\partial \Omega_Q}  G^{k}_{q,\eta}(x-y)\mu(y)\,d\sigma_y \qquad \forall x \in \mathbb{R}^n.
\]
Moreover we set 
\[
\left(\mathcal{K}^{k}_{q,\eta}\right)^*[\partial\Omega_Q,\mu](x) := \int_{\partial \Omega_Q}\nu_{\Omega_Q}(x) \cdot\nabla G^{k}_{q,\eta}(x-y)\mu(y)\,d\sigma_y \qquad \forall x \in \partial\Omega_Q.
\]
The proof of the following properties of $\mathcal{S}^{k}_{q,\eta}[\partial\Omega_Q,\mu]$ follows the lines of the proof of the previous Proposition \ref{prop:dlp}, that is it uses Proposition \ref{qnSR} together with the known properties of the single layer potential associated with the fundamental solution $S_n(\cdot,k)$.  
\begin{prop}\label{prop:slp}
Let $\alpha \in \mathopen]0,1[$, $q \in \mathbb{D}^+_n(\mathbb{R})$, $\eta \in \mathbb{R}^n$, and $k \in \mathbb{C}$.  Let $\Omega_Q$ be a bounded open subset of $\mathbb{R}^n$ of class $C^{1,\alpha}$  such that $\overline{\Omega_Q} \subseteq Q$.  
Let $\mu \in C^{0,\alpha}(\partial \Omega_Q)$. Then the following statements hold.
\begin{itemize}
\item[i)]  $\mathcal{S}^{k}_{q,\eta}[\partial\Omega_Q,\mu]$ is of class $C^\infty(\mathbb{S}_q[\Omega_Q] \cup \mathbb{S}^-_q[\Omega_Q])$ and
\begin{align*}
(\Delta&+k^2)\mathcal{S}^{k}_{q,\eta}[\partial\Omega_Q,\mu](x) \\
&= -\frac{1}{|Q|_n}\sum_{z \in Z_{q,\eta}(k)}e^{ix\cdot (2\pi  q^{-1}z+\eta)} \int_{\partial \Omega_Q}e^{-iy\cdot (2\pi  q^{-1}z+\eta)} \mu(y)\,d\sigma_y\\
&\hspace{9cm}\forall x \in \mathbb{S}_q[\Omega_Q] \cup \mathbb{S}^-_q[\Omega_Q].
\end{align*}
\item[ii)] $\mathcal{S}^{k}_{q,\eta}[\partial\Omega_Q,\mu]$ is {$(Q,\eta)$-quasi-periodic}.
\item[iii)] $\mathcal{S}^{k}_{q,\eta}[\partial\Omega_Q,\mu]$ is continuous in $\mathbb{R}^n$ and 
 \begin{align*}
&\mathcal{S}^{k,+}_{q,\eta}[\partial\Omega_Q,\mu] := \mathcal{S}^{k}_{q,\eta}[\partial\Omega_Q,\mu]_{|\overline{\mathbb{S}_q[\Omega_Q]}} \in C^{1,\alpha}(\overline{\mathbb{S}_q[\Omega_Q]}),\\
 &\mathcal{S}^{k,-}_{q,\eta}[\partial\Omega_Q,\mu] := \mathcal{S}^{k}_{q,\eta}[\partial\Omega_Q,\mu]_{|\overline{\mathbb{S}^-_q[\Omega_Q]}} \in C^{1,\alpha}(\overline{\mathbb{S}^-_q[\Omega_Q]}).
 \end{align*}
  Moreover:
\begin{align} \label{jfs}
&\nu_{\Omega_Q}(x) \cdot \nabla\mathcal{S}^{k,\pm}_{q,\eta}[\partial\Omega_Q,\mu](x) = \mp\frac{1}{2}\mu(x)+\left(\mathcal{K}^{k}_{q,\eta}\right)^*[\partial\Omega_Q,\mu](x) \qquad \forall x \in \partial\Omega_Q.
\end{align}
\item[iv)]  The map from $C^{0,\alpha}(\partial \Omega_Q)$ to $ C^{1,\alpha}(\overline{\mathbb{S}_q[\Omega_Q]})$ 
that takes $\mu$ to $\mathcal{S}^{k,+}_{q,\eta}[\partial\Omega_Q,\mu]$ and the map from $C^{0,\alpha}(\partial \Omega_Q)$ to $ C^{1,\alpha}(\overline{\mathbb{S}^-_q[\Omega_Q]})$ 
that takes $\mu$ to $\mathcal{S}^{k,-}_{q,\eta}[\partial\Omega_Q,\mu]$ are linear and continuous.
\item[v)] The map that takes $\mu \in C^{0,\alpha}(\partial\Omega_Q)$ to $\left(\mathcal{K}^{k}_{q,\eta}\right)^*[\partial\Omega_Q,\mu]$ is a compact operator from  $C^{0,\alpha}(\partial\Omega_Q)$ to itself.
\end{itemize}
\end{prop}

\section{Basic boundary value problems}\label{s:qperbvp}
Throughout this section we fix $\alpha \in \mathopen]0,1[$, $q \in \mathbb{D}^+_n(\mathbb{R})$, $\eta \in \mathbb{R}^n$, and $k \in \mathbb{C}$.  Moreover, we fix a bounded open subset
$\Omega_Q$   of $\mathbb{R}^n$ of class $C^{1,\alpha}$  such that $\overline{\Omega_Q} \subseteq Q$.

We start recalling some known facts regarding the spectrum of the Laplacian. 
It is well-known that the spectrum of the Laplacian $-\Delta$ acting on functions which are periodic with respect to $Q$ can be seen as the spectrum of the Laplacian on the 
flat torus $\mathbb{\mathbb{R}}^n / q\mathbb{Z}^n$, it
 is made of eigenvalues, and it is given by 
\[
\sigma_{q,0}(-\Delta)  := \left\{|2\pi q^{-1}z|^2 : z \in \mathbb{Z}^n\right\}.
\]
Moreover, if $\lambda \in \sigma_{q,0}(-\Delta) $ its multiplicity coincides with 
\[
\#\left\{z \in \mathbb{Z}^n: \lambda = |2\pi q^{-1}z|^2\right\}.
\]
For more details, we refer to, e.g.,   Chavel \cite[Ch. II]{Ch84}.

Similarly, the spectrum of the  Laplacian $-\Delta$ acting on $(Q,\eta)$-quasi-periodic functions is made of eigenvalues{, is} given by
\[
\sigma_{q,\eta}(-\Delta) := \left\{|2\pi q^{-1}z+\eta|^2 : z \in \mathbb{Z}^n\right\}{,}
\]
and if $\lambda \in\sigma_{q,\eta}(-\Delta)$ its multiplicity coincides with 
\[
\#\left\{z \in \mathbb{Z}^n: \lambda = |2\pi q^{-1}z+\eta|^2\right\}.
\]
In other words, if $k \in \mathbb{C}$ then $k^2 \in \sigma_{q,\eta}(-\Delta)$ if and only if $Z_{q,\eta}(k) \neq \emptyset$.

Next we consider the spectrum of the  Dirichlet Laplacian on $(Q,\eta)$-quasi-periodic functions of $\mathbb{S}^-_q[\Omega_Q]$. Namely, we say that
$\lambda \in \mathbb{C}$ is a Dirichlet  $(Q,\eta)$-quasi-periodic eigenvalue of $-\Delta$ on $\mathbb{S}^-_q[\Omega_Q]$, and we write 
$\lambda \in \sigma^D_{q,\eta}(-\Delta, \mathbb{S}^-_q[\Omega_Q]) $, if there exists a non-zero solution $u$ of 
\begin{equation*}
\begin{cases}
-\Delta u = \lambda u \qquad &\mbox{ in } \mathbb{S}^-_q[\Omega_Q],\\
u \, \mbox{is $(Q,\eta)$-quasi-periodic},\\
u=0 \qquad &\mbox{ on } \partial\Omega_Q.
\end{cases}
\end{equation*}
By classical spectral theory, the spectrum is discrete and made of real positive eigenvalues  of finite multiplicity that can be arranged in a diverging sequence.
 Similarly, we say that
$\lambda \in \mathbb{C}$ is a Neumann  $(Q,\eta)$-quasi-periodic eigenvalue of $-\Delta$ on $\mathbb{S}^-_q[\Omega]$, and we write 
$\lambda \in \sigma^N_{q,\eta}(-\Delta, \mathbb{S}^-_q[\Omega_Q])$, if there exists a non-zero solution $u$ of 
\begin{equation*}
\begin{cases}
-\Delta u = \lambda u \qquad &\mbox{ in } \mathbb{S}^-_q[\Omega_Q],\\
u \, \mbox{is $(Q,\eta)$-quasi-periodic},\\
\partial_{\nu_{\Omega_Q}} u=0 \qquad &\mbox{ on } \partial\Omega_Q.
\end{cases}
\end{equation*}
Again, by classical spectral theory, the spectrum is discrete and made of real {non-negative} eigenvalues of finite multiplicity that can be arranged in a diverging sequence.
Finally, we respectively denote by $\sigma^D(-\Delta, \Omega_Q)$ and $\sigma^N(-\Delta, \Omega_Q)$ the set 
of Dirichlet and Neumann eigenvalues of the Laplacian on {$\Omega_Q$}.

Now we pass to consider some basic boundary value problems for the $(Q,\eta)$-quasi-periodic Helmholtz equation, i.e., boundary value problems for the  Helmholtz equation with $(Q,\eta)$-quasi-periodicity conditions. More precisely, we consider the Dirichlet and Neumann problems.

In the following sections we will always assume that $k^2 \notin \sigma_{q,\eta}(-\Delta)$, or, equivalently,  $Z_{q,\eta}(k) = \emptyset$. So, we will always have that the double and single  layer potentials solve the Helmholtz equation. That is,
\[
\begin{aligned}
&(\Delta+k^2)\mathcal{D}^{k}_{q,\eta}[\partial\Omega_Q,\phi](x)=0\\
&(\Delta+k^2)\mathcal{S}^{k}_{q,\eta}[\partial\Omega_Q,\psi](x)=0 
\end{aligned}
\] 
for all $x\in \mathbb{S}_q[\Omega_Q] \cup \mathbb{S}^-_q[\Omega_Q]$, $\phi\in C^{1,\alpha}(\partial \Omega_Q)$, and $\psi\in C^{0,\alpha}(\partial \Omega_Q)$ (cf.~Propositions \ref{prop:dlp} and \ref{prop:slp}).

\subsection{Dirichlet problem}
Let $g \in C^{1,\alpha}(\partial \Omega_Q)$ and $k \in \mathbb{C}$. In this subsection we consider the Dirichlet problem
\begin{equation}\label{pb:dir}
\begin{cases}
\Delta u + k^2 u =0\qquad &\mbox{ in } \mathbb{S}^-_q[\Omega_Q],\\
u \, \mbox{is $(Q,\eta)$-quasi-periodic},\\
u=g \qquad &\mbox{ on } \partial\Omega_Q.
\end{cases}
\end{equation}
In the next theorem we show how to solve problem \eqref{pb:dir} and how the solution can be represented 
by means of layer potentials under suitable assumptions on the wave number $k^2$. To this aim, we find convenient to set 
\[
A(k) := \begin{cases}
1 \quad \mbox{ if } k^2  \in \sigma^N(-\Delta, \Omega_Q),\\
0 \quad \mbox{ otherwise }.
\end{cases}
\] 

\begin{thm}\label{thm:dirpb}
Let $\alpha \in \mathopen]0,1[$, $q \in \mathbb{D}^+_n(\mathbb{R})$, and $\eta \in \mathbb{R}^n$.
Let $\Omega_Q$ be a bounded open subset of $\mathbb{R}^n$ of class $C^{1,\alpha}$  such that $\overline{\Omega_Q} \subseteq Q$.  
Let {$k \in \mathbb{C}$ be such} that $k^2 \notin\sigma_{q,\eta}(-\Delta)$, 
$k^2 \notin \sigma^D_{q,\eta}(-\Delta, \mathbb{S}^-_q[\Omega_Q])$. Then the following statements hold.
\begin{itemize}
\item[i)] The integral operator $\mathcal{T}$ from $C^{1,\alpha}(\partial\Omega_Q)$ to itself defined by
\[
\mathcal{T} := -\frac{1}{2}\mathbb{I} + \mathcal{K}^{k}_{q,\eta}[\partial\Omega_Q,\cdot] +i A(k)\mathcal{S}^{k}_{q,\eta}[\partial\Omega_Q,\cdot]_{|\partial\Omega_Q}\,,
\]
where $\mathbb{I}$ is the identity operator on  $C^{1,\alpha}(\partial\Omega_Q)$,  is a linear homeomorphism.
\item[ii)] Let $g \in C^{1,\alpha}(\partial \Omega_Q)$. Then problem \eqref{pb:dir} admits a unique solution
 $u \in C^{1,\alpha}(\overline{\mathbb{S}^-_q[\Omega_Q]})$. Moreover
 \[
 u = \mathcal{D}^{k,-}_{q,\eta}[\partial\Omega_Q,\mu] +i A(k)\mathcal{S}^{k,-}_{q,\eta}[\partial\Omega_Q,\mu]\, ,
 \]
 where
 \[
 \mu = \mathcal{T}^{(-1)}[g].
 \]
\end{itemize}
\end{thm}
\begin{proof}
We first consider statement i). By Proposition \ref{prop:dlp} v), Proposition \ref{prop:slp} iv), by the compactness of the embedding 
of $C^{1,\alpha}(\partial \Omega_Q)$ in $C^{0,\alpha}(\partial \Omega_Q)$, and by the continuity of the restriction operator from  $ C^{1,\alpha}(\overline{\mathbb{S}^-_q[\Omega_Q]})$ to $C^{1,\alpha}(\partial \Omega_Q)$, the operator 
\[
\mu \mapsto \mathcal{K}^{k}_{q,\eta}[\partial\Omega_Q,\mu] +i A(k)\mathcal{S}^{k}_{q,\eta}[\partial\Omega_Q,\mu]_{|\partial\Omega_Q}
\]
is compact in $C^{1,\alpha}(\partial \Omega_Q)$. Therefore $\mathcal{T}$ is a Fredholm operator of index $0$ and, accordingly, to show that $\mathcal{T}$ is invertible it suffices to show that  it is injective. Let $\mu \in C^{1,\alpha}(\partial \Omega_Q)$ be such that 
\[
\mathcal{T}[\mu] = -\frac{1}{2}\mu + \mathcal{K}^{k}_{q,\eta}[\partial\Omega_Q,\mu] -i  A(k)\mathcal{S}^{k}_{q,\eta}[\partial\Omega_Q,\mu]_{|\partial\Omega_Q}=0.
\]
We now consider separately two cases. We first suppose that $ k^2  \notin \sigma^N(-\Delta,\Omega_Q)$, a case in which $A(k)=0$.   By the jump formula  \eqref{jfd} for the double layer potential,  the function defined by 
\[
u := \mathcal{D}^{k,-}_{q,\eta}[\partial\Omega_Q,\mu]-i  A(k)\mathcal{S}^{k,-}_{q,\eta}[\partial\Omega_Q,\mu]  = \mathcal{D}^{k,-}_{q,\eta}[\partial\Omega_Q,\mu] \qquad \mbox{ in } \overline{\mathbb{S}^-_q[\Omega_Q]}
\]
solves 
\begin{equation*}
\begin{cases}
\Delta u + k^2 u =0 \qquad &\mbox{ in } \mathbb{S}^-_q[\Omega_Q],\\
u \, \mbox{is $(Q,\eta)$-quasi-periodic},\\
u=0 \qquad &\mbox{ on } \partial\Omega_Q.
\end{cases}
\end{equation*}
Note that $u$ satisfies the Helmholtz equation since by assumption $k^2 \notin \sigma_{q,\eta}(-\Delta)$, which is equivalent to $Z_{q,\eta}(k) = \emptyset$, and accordingly the  double layer potential satisfies the Helmholtz equation (see Proposition \ref{prop:dlp} i)). 
Since by assumption $k^2 \notin \sigma^D_{q,\eta}(-\Delta, \mathbb{S}^-_q[\Omega_Q]) $, then 
\[
u=0\qquad \mbox{ in } \overline{\mathbb{S}^-_q[\Omega_Q]}.
\]
Now we set 
\[
v = \mathcal{D}^{k,+}_{q,\eta}[\partial\Omega_Q,\mu]\qquad \mbox{ in }  \overline{\Omega_Q}.
\]
By the continuity of the normal derivative of the interior and exterior double layer potential in Proposition \ref{prop:dlp} iii), the function 
$v$ solves the Neumann problem
\begin{equation*}
\begin{cases}
\Delta v + k^2 v =0 \qquad &\mbox{ in } \Omega_Q,\\
\partial_{\nu_{\Omega_Q}} v=0 \qquad &\mbox{ on } \partial\Omega_Q.
\end{cases}
\end{equation*}
Since $k^2$ is not an eigenvalue of the Neumann Laplacian  in $\Omega_Q$, then 
\[
v=0\qquad \mbox{ in } \overline{\Omega_Q}.
\]
Thus, by the jump formula for the double layer potential in Proposition \ref{prop:dlp} iii):
\[
\mu =  \mathcal{D}^{k,+}_{q,\eta}[\partial\Omega_Q,\mu]_{|\partial\Omega_Q}  -  \mathcal{D}^{k,-}_{q,\eta}[\partial\Omega_Q,\mu]_{|\partial\Omega_Q} = v-u =0 \quad \mbox{ on } \partial\Omega_Q.
\]
 
Now we consider the case in which $ k^2  \in \sigma^N(-\Delta,\Omega_Q)$.  Again, by the jump formula for the double layer potential and by the continuity of the single layer potential, the function defined by 
\[
u := \mathcal{D}^{k,-}_{q,\eta}[\partial\Omega_Q,\mu] -i  \mathcal{S}^{k,-}_{q,\eta}[\partial\Omega_Q,\mu]  \qquad \mbox{ in } \overline{\mathbb{S}^-_q[\Omega_Q]}
\]
solves 
\begin{equation*}
\begin{cases}
\Delta u + k^2 u =0 \qquad &\mbox{ in } \mathbb{S}^-_q[\Omega_Q],\\
u \, \mbox{is $(Q,\eta)$-quasi-periodic},\\
u=0 \qquad &\mbox{ on } \partial\Omega_Q.
\end{cases}
\end{equation*}
Since, by assumption, $k^2 \notin \sigma^D_{q,\eta}(-\Delta, \mathbb{S}^-_q[\Omega_Q]) $ then 
\[
u=0\qquad \mbox{ in } \overline{\mathbb{S}^-_q[\Omega_Q]}.
\]
Now we set
\[
v = \mathcal{D}^{k,+}_{q,\eta}[\partial\Omega_Q,\mu] -i \mathcal{S}^{k,+}_{q,\eta}[\partial\Omega_Q,\mu]  \qquad \mbox{ in }  \overline{\Omega_Q}.
\]
 By the jumping properties of the double layer potential (equality \eqref{jfd}), by the continuity of the single layer potential (Proposition \ref{prop:slp} iii)), and by equality $u=0$, we see that
\begin{equation}\label{dirpb1}
v_{|\partial\Omega_Q} = \mu +  \mathcal{D}^{k,-}_{q,\eta}[\partial\Omega_Q,\mu]_{|\partial\Omega_Q}  -i  \mathcal{S}^{k,-}_{q,\eta}[\partial\Omega_Q,\mu]_{|\partial\Omega_Q} = \mu+u= \mu.
\end{equation}
Moreover, by the continuity of the normal derivative of the double layer potential \eqref{jfd}, by the jump formula for the normal derivative of the single layer potential \eqref{jfs}, and by equality $u=0$, we have
\begin{equation}\label{dirpb2}
\frac{\partial}{\partial\nu_{\Omega_Q}}v = \frac{\partial}{\partial\nu_{\Omega_Q}} \mathcal{D}^{k,-}_{q,\eta}[\partial\Omega_Q,\mu]-i\left(\frac{\partial}{\partial\nu_{\Omega_Q}} \mathcal{S}^{k,-}_{q,\eta}[\partial\Omega_Q,\mu]-\mu\right)= \frac{\partial u}{\partial\nu_{\Omega_Q}}+i\mu=i\mu
\end{equation}
By the $(Q,\eta)$-quasi-periodicity  of $v$, and since $\nu_Q$ has opposite sign on opposite faces of $\partial Q$, we can verify that
\begin{equation}\label{dirpb3}
\int_{\partial Q}\frac{\partial v}{\partial\nu_{Q}} \overline{v}\,d\sigma = \int_{\partial Q}\left(\frac{\partial (v(x) e^{-i\eta\cdot x})}{\partial\nu_{Q}}+i\eta \cdot\nu_Q(x)v(x)e^{-i\eta\cdot x }\right) \overline{v(x)e^{-i\eta\cdot x }}\,d\sigma_x={0}
\end{equation}
(see also Ammari, Kang and Lee \cite[p. 125]{AmKaLe09}).
Then, the first Green identity (cf., e.g., Colton and Kress \cite[(3.4), p. 68]{CoKr13}) and equalities 
\eqref{dirpb1}, \eqref{dirpb2}, and \eqref{dirpb3}, 
imply that
\begin{align}\label{eq:varexist:1}
i\int_{\partial\Omega_Q}|\mu|^2\,d\sigma = \int_{\partial\Omega_Q}\frac{\partial v}{\partial\nu_{\Omega_Q}} \overline{v}\,d\sigma &=- \int_{\partial Q}\frac{\partial v}{\partial\nu_Q} \overline{v}\,d\sigma+\int_{\partial\Omega_Q}\frac{\partial v}{\partial\nu_{\Omega_Q}} \overline{v}\,d\sigma \\ \nonumber
&={-} \int_{Q\setminus \overline\Omega_Q}|\nabla v|^2-k^2|v|^2\,dx.
\end{align}
Taking the imaginary part in \eqref{eq:varexist:1} and recalling that $k^2$ is an eigenvalue of the Neumann Laplacian and thus a real number, we get
\begin{equation*}
 \int_{\partial \Omega_Q}|\mu|^2d\sigma=0
\end{equation*}
 and statement i) follows. 

The validity of statement ii) follows from statement i) and from the properties of layer potentials (see Propositions \ref{prop:dlp} and \ref{prop:slp}).
\end{proof}

\begin{rem}
To prove Theorem \ref{thm:dirpb} we have adjusted an argument that was used by Colton and Kress  \cite[Thm.~3.33, p.~91]{CoKr13} to obtain a similar result for the classical Helmholtz equation. Provided that  $k^2 \notin \sigma_{q,\eta}(-\Delta)$ and
$k^2 \notin \sigma^D_{q,\eta}(-\Delta, \mathbb{S}^-_q[\Omega_Q])$, Theorem \ref{thm:dirpb} shows that the solution $u$ of problem \eqref{pb:dir} can be written as a sum of a double layer potential and a single layer potential. If we further assume that $k^2 \notin \sigma^N(-\Delta,\Omega_Q)$, then  $A(k) =0$ and a double layer potential is sufficient. Indeed, in that  case  we have
\[
 u = \mathcal{D}^{k,-}_{q,\eta}[\partial\Omega_Q,\mu]
\]
where $\mu$ is the unique solution of the integral equation
\[
 -\frac{1}{2}\mu+ \mathcal{K}^{k}_{q,\eta}[\partial\Omega_Q,\mu] = g.
\]
\end{rem}

As an immediate consequence of Theorem \ref{thm:dirpb} we obtain the following representation result for a $(Q,\eta)$-quasi-periodic function satisfying the  Helmholtz equation.
\begin{cor}\label{cor:repdir}
Let $\alpha \in \mathopen]0,1[$, $q \in \mathbb{D}^+_n(\mathbb{R})$, and $\eta \in \mathbb{R}^n$.
Let $\Omega_Q$ be a bounded open subset of $\mathbb{R}^n$ of class $C^{1,\alpha}$  such that $\overline{\Omega_Q} \subseteq Q$.  
Let {$k \in \mathbb{C}$ be such} that $k^2 \notin\sigma_{q,\eta}(-\Delta)$, 
$k^2 \notin \sigma^D_{q,\eta}(-\Delta, \mathbb{S}^-_q[\Omega_Q])$. Let $u \in C^{1,\alpha}(\overline{\mathbb{S}^-_q[\Omega_Q]})$ be such that
\[
\begin{cases}
\Delta u + k^2 u =0 \qquad &\mbox{ in } \mathbb{S}^-_q[\Omega_Q],\\
u \, \mbox{is $(Q,\eta)$-quasi-periodic}.
\end{cases}
\]
Then there exists a unique $\mu \in C^{1,\alpha}(\partial\Omega_Q)$ such that
\[
u =  \mathcal{D}^{k,-}_{q,\eta}[\partial\Omega_Q,\mu] +i A(k)\mathcal{S}^{k,-}_{q,\eta}[\partial\Omega_Q,\mu],
\]
where $\mu$ is the unique solution in $C^{1,\alpha}(\partial\Omega_Q)$ of 
\[
\mathcal{T}[\mu]= u_{|\partial\Omega_Q}
\]
\end{cor}

\subsection{Neumann problem}
Let $h \in C^{0,\alpha}(\partial \Omega_Q)$ and $k \in \mathbb{C}$. Here we consider the Neumann problem
\begin{equation}\label{pb:neu}
\begin{cases}
\Delta u + k^2 u =0 \qquad &\mbox{ in } \mathbb{S}^-_q[\Omega_Q],\\
u \, \mbox{is $(Q,\eta)$-quasi-periodic},\\
\partial_{\nu_{\Omega_Q}} u=h \qquad &\mbox{ on } \partial\Omega_Q.
\end{cases}
\end{equation}
In the next theorem we show how to solve problem \eqref{pb:neu} and how the solution can be represented 
by means of layer potentials.  For the sake of simplicity we require the additional assumption that $k^2$ is not an eigenvalue of the Dirichlet Laplacian in $\Omega_Q$. In this case it is possible to represent the solution by means of only a single layer potential. Note that to our scope, this requirement is sufficient  (see Section \ref{s:singbvp} and in particular \eqref{assepsbis})
\begin{thm}\label{thm:neupb}
Let $\alpha \in \mathopen]0,1[$, $q \in \mathbb{D}^+_n(\mathbb{R})$, and $\eta \in \mathbb{R}^n$.
Let $\Omega_Q$ be a bounded open subset of $\mathbb{R}^n$ of class $C^{1,\alpha}$  such that $\overline{\Omega_Q} \subseteq Q$.  
Let {$k \in \mathbb{C}$ be such} that $k^2 \notin \sigma_{q,\eta}(-\Delta)$, 
$k^2 \notin  \sigma^N_{q,\eta}(-\Delta, \mathbb{S}^-_q[\Omega_Q])$, $k^2 \notin \sigma^{D}(-\Delta,\Omega_Q)$. Then the following statements hold.
\begin{itemize}
\item[i)] The integral operator $\mathcal{M}$ from $C^{0,\alpha}(\partial\Omega_Q)$ to itself defined by
\[
\mathcal{M} := \frac{1}{2}\mathbb{I} + \left(\mathcal{K}^{k}_{q,\eta}\right)^*[\partial\Omega_Q,\cdot]  ,
\]
where $\mathbb{I}$ is the identity operator on  $C^{0,\alpha}(\partial\Omega_Q)$,  is a linear homeomorphism.
\item[ii)] Let $h \in C^{0,\alpha}(\partial \Omega_Q)$. Then problem \eqref{pb:neu} admits a unique solution
 $u \in C^{1,\alpha}(\overline{\mathbb{S}^-_q[\Omega]})$. Moreover
 \[
 u = \mathcal{S}^{k,-}_{q,\eta}[\partial\Omega_Q,\mu]
 \]
 where
 \[
 \mu = \mathcal{M}^{(-1)}[g].
 \]
\end{itemize}
\end{thm}
\begin{proof}
We consider statement i). By Proposition   \ref{prop:slp} v), the map $\left(\mathcal{K}^{k}_{q,\eta}\right)^*[\partial\Omega_Q,\cdot]$ is compact in $C^{0,\alpha}(\partial \Omega_Q)$. Therefore $\mathcal{M}$ is a Fredholm operator of index $0$ and, accordingly, to show that $\mathcal{M}$ is invertible it suffices to show that  it is injective. Let $\mu \in C^{0,\alpha}(\partial \Omega_Q)$ be such that 
\[
\mathcal{M}[\mu] = \frac{1}{2}\mu + \left(\mathcal{K}^{k}_{q,\eta}\right)^*[\partial\Omega_Q,\mu] =0.
\]
 By the jump formula \eqref{jfs} for the normal derivative of the layer potential,  the function defined by 
\[
u :=  \mathcal{S}^{k,-}_{q,\eta}[\partial\Omega_Q,\mu]   \qquad \mbox{ in } \overline{\mathbb{S}^-_q[\Omega_Q]}
\]
solves 
\begin{equation*}
\begin{cases}
\Delta u + k^2 u =0 \qquad &\mbox{ in } \mathbb{S}^-_q[\Omega_Q],\\
u \, \mbox{is $(Q,\eta)$-quasi-periodic},\\
\partial_{\nu_{\Omega_Q}} u=0 \qquad &\mbox{ on } \partial\Omega_Q.
\end{cases}
\end{equation*}
The first equation of the system follows by the assumption that $k^2 \notin \sigma_{q,\eta}(-\Delta)$, which is equivalent to $Z_{q,\eta}(k) = \emptyset$ and, by Proposition \ref{prop:slp} i),  implies that  the  single layer potential satisfies the Helmholtz equation. 
Since by assumption $k^2 \notin \sigma^N_{q,\eta}(-\Delta, \mathbb{S}^-_q[\Omega_Q])$, we have
\[
u=0\qquad \mbox{ in } \overline{\mathbb{S}^-_q[\Omega_Q]}.
\]
Now we set 
\[
v = \mathcal{S}^{k,+}_{q,\eta}[\partial\Omega_Q,\mu]\qquad \mbox{ in }  \overline{\Omega_Q}.
\]
By the continuity of  the single layer potential through the boundary (see Proposition \ref{prop:slp} iii)), the function 
$v$ solves the Dirichlet problem
\begin{equation*}
\begin{cases}
\Delta {v} + k^2 {v} =0 \qquad &\mbox{ in } \Omega_Q,\\
 v=0 \qquad &\mbox{ on } \partial\Omega_Q.
\end{cases}
\end{equation*}
Since $k^2$ is not an eigenvalue of the Dirichlet Laplacian  in $\Omega_Q$, then 
\[
v=0\qquad \mbox{ in } \overline{\Omega_Q}.
\]
Then, by the jump formula for the normal derivative of the single layer potential in Proposition \ref{prop:slp} iii):
\begin{align*}
\mu &=  \partial_{\nu_{\Omega_Q}}\mathcal{S}^{k,+}_{q,\eta}[\partial\Omega_Q,\mu]  - \partial_{\nu_{\Omega_Q}}\mathcal{S}^{k,-}_{q,\eta}[\partial\Omega_Q,\mu] \\
&=  \partial_{\nu_{\Omega_Q}} v- \partial_{\nu_{\Omega_Q}} u =0,
\end{align*}
and accordingly the statement follows.
\end{proof}
Similarly to the case of the Dirichlet problem, the previous Theorem \ref{thm:neupb} implies the following representation result.
\begin{cor}\label{cor:repneu}
Let $\alpha \in \mathopen]0,1[$, $q \in \mathbb{D}^+_n(\mathbb{R})$, and $\eta \in \mathbb{R}^n$.
Let $\Omega_Q$ be a bounded open subset of $\mathbb{R}^n$ of class $C^{1,\alpha}$  such that $\overline{\Omega_Q} \subseteq Q$.  
Let {$k \in \mathbb{C}$ be such} that $k^2 \notin \sigma_{q,\eta}(-\Delta)$, 
$k^2 \notin  \sigma^N_{q,\eta}(-\Delta, \mathbb{S}^-_q[\Omega_Q])$, $k^2 \notin \sigma^{D}(-\Delta,\Omega_Q)$. Let $u \in C^{1,\alpha}(\overline{\mathbb{S}^-_q[\Omega_Q]})$ be such that
\[
\begin{cases}
\Delta u + k^2 u =0 \qquad &\mbox{ in } \mathbb{S}^-_q[\Omega_Q],\\
u \, \mbox{is $(Q,\eta)$-quasi-periodic}.
\end{cases}
\]
Then there exists a unique $\mu \in C^{0,\alpha}(\partial\Omega_Q)$ such that
\[
u = \mathcal{S}^{k,-}_{q,\eta}[\partial\Omega_Q,\mu],
\]
where $\mu$ is the unique solution in $C^{0,\alpha}(\partial\Omega_Q)$ of 
\[
\mathcal{M}[\mu]= \frac{\partial}{\partial \nu_{\Omega_Q}}u{\, .}
\]
\end{cor}

\section{Singular perturbations for quasi-periodic layer potentials for the Helmholtz equation}\label{s:singpot}

In this section, we study the behavior of quasi-periodic layer potentials upon singular domain perturbations. More precisely, we consider quasi-periodic layer potentials supported on the boundary of a set of the type $\Omega_{p,\epsilon}:=p+\epsilon \Omega$ where $p$ {belongs to} $Q$, $\epsilon$ {is} a sufficiently small positive parameter, and $\Omega$ {is} a sufficiently regular bounded open set. We are interested into representation formulas for the layer potentials in terms of real analytic operators when $\epsilon$ is close to the degenerate value $0$, in correspondence of which the set collapses to {the point $p$}. These results will be fundamental to study singularly perturbed quasi-periodic boundary value problems for the Helmholtz equation in the set $\mathbb{S}_q^-[\Omega_{p,\epsilon}]$ as $\epsilon \to 0^+$ by means of quasi-periodic layer potentials. We observe that the results of the present section can be seen as the quasi-periodic analog of some of the properties studied in \cite{DaLuMu22} on (singular and regular) domain perturbations for classical layer potentials for the Laplace equation. Asymptotic formulas for layer potentials are available in specific dimensions and geometric settings in Ammari, Kang, and Lee~\cite[Lem.~3.3]{AmKaLe09},  Feppon and Ammari~\cite[Prop.~2.3]{FeAm22},\cite[Prop.~2.5]{FeAm21}.

\subsection{The geometric setting}

Let $n\in \mathbb{N}\setminus \{0,1\}$.  Let $\alpha\in\mathopen]0,1[$. We take a subset  $\Omega$  of $\mathbb{R}^n$ satisfying the following  assumption:
	\begin{equation}\label{Omega_def}
	\begin{array}{c}
	\Omega\,\,\mbox{is a bounded open connected subset of}\,\,\mathbb{R}^n\,\,\mbox{of class}\,C^{1,\alpha} \\
	\mbox{such that}\   \mathbb{R}^n \setminus{\overline{\Omega}}\,\,\mbox{is connected.}
	\end{array}
	\end{equation}
Let $p\in Q$. Then there exists $\epsilon_0\in\mathopen]0,+\infty[$ such that
	\begin{equation}\label{epsilon_0}
	  p+\epsilon\,{\overline{\Omega}}\subseteq Q \quad\forall\epsilon\in\mathopen]-\epsilon_0, \epsilon_0[\, .
	\end{equation}	
To shorten our notation, we set
	$$\Omega_{p,\epsilon}:= p+\epsilon\Omega\quad\forall\epsilon\in\mathbb{R}\, .$$

%

\subsection{Notation and preliminaries}\label{hlm:ss:notation}

In this subsection, as a first step, we deduce some rescaling formulas for the family of fundamental solutions $S_n(\cdot, k)$ of the differential operators $\Delta +k^2$ with $k \in \mathbb{C}$. We retain the notation introduced in Subsection \ref{s:holfundsol}. However, for our specific purpose, we need also to introduce some other notation.

Let $\mathcal{J}_n$ be the function from $\mathbb{C}$ to $\mathbb{C}$ introduced in Definition \ref{def:hlmups}. Then we define the function $T^k_n$ from $\mathbb{R}^n$ to $\mathbb{C}$ by setting
\[
T^k_n(x):= \mathcal{J}_n(k|x|) \qquad \forall x \in \mathbb{R}^n.
\]
Then $T^k_n$ is a real analytic function (see the proof of Lanza de Cristoforis and Rossi \cite[Prop.~3.3]{LaRo08}). Moreover, if $n$ is odd, then $T^k_n(x)=0$ for all $x \in \mathbb{R}^n$. 
Let $\Upsilon_n$ be the function defined in Definition \ref{def:hlmups}. We note that if $\epsilon >0$ then
\begin{align*}
\Upsilon_n(\epsilon r,k)&=  k^{n-2}\mathcal{J}_n(\epsilon rk)\log( \epsilon r)+\frac{\mathcal{N}_n(\epsilon rk)}{\epsilon^{n-2}r^{n-2}}\\
&=  k^{n-2}\mathcal{J}_n(\epsilon rk)\log \epsilon +\frac{1}{\epsilon^{n-2}} \bigg( (\epsilon k)^{n-2}\mathcal{J}_n( r \epsilon k) \log r+\frac{\mathcal{N}_n( r \epsilon k)}{r^{n-2}}\bigg)\, ,
\end{align*}
for all $(r,k) \in \mathopen]0,+\infty\mathclose[\times \mathbb{C}$.  Therefore, if $\epsilon >0$ and $x \in \mathbb{R}^n \setminus \{0\}$, a straightforward computation shows that
\begin{equation}\label{hlm:n:eq:Snk}
S_n(\epsilon x,k)=\frac{1}{\epsilon^{n-2}}S_n(x,\epsilon k)+ (\log \epsilon) k^{n-2} T^k_n(\epsilon x),
\end{equation} 
and
\begin{equation}\label{hlm:n:eq:DSnk}
\nabla S_n(\epsilon x,k)=\frac{1}{\epsilon^{n-1}}\nabla S_n(x,\epsilon k)+ (\log \epsilon) k^{n-2} \nabla T^k_n(\epsilon x).
\end{equation} 

\subsection{Singular perturbations for the quasi-periodic single layer potential for the Helmholtz equation} 

In this subsection, we consider the behavior of the quasi-periodic single layer potential and of an associated operator upon singular domain perturbations. We begin by studying the behavior of the quasi-periodic single layer potential restricted to the boundary of $\Omega_{p,\epsilon}=p+\epsilon \Omega$. {In order to work with functional spaces that do not depend on the perturbation parameter, we pull-back the single layer to the fixed domain $\partial \Omega$ and we push-forward a density defined on $\partial\Omega$ to $\partial\Omega_{p,\epsilon}$}.
 More precisely, we consider the map from $C^{0,\alpha}(\partial \Omega)$ to $C^{1,\alpha}(\partial \Omega)$ that takes a density $\theta$ to the function  
\[
\mathcal{S}^{k}_{q,\eta}[\partial\Omega_{p,\epsilon},\theta((\cdot-p)/\epsilon)](p+\epsilon t)\quad\forall t\in\partial\Omega\,.
\]
 
\begin{prop}\label{hlm:smp:prop:bd1}
Let $\alpha \in \mathopen]0,1[$, $q \in \mathbb{D}^+_n(\mathbb{R})$, $\eta \in \mathbb{R}^n$. Let $k \in \mathbb{C}$. Let $\Omega$ be as in assumption \eqref{Omega_def}. Let $p \in Q$. Let $\epsilon_0$ be as in assumption \eqref{epsilon_0}. Let $M_1$, $M_2$, $M_3$ be the maps from $\mathopen]-\epsilon_0,\epsilon_0[$ to $\mathcal{L}(C^{0,\alpha}(\partial \Omega),C^{1,\alpha}(\partial \Omega))$, defined by
\begin{align*}
& M_1[\epsilon](\theta)(t):= \int_{\partial \Omega}S_n(t-s,\epsilon k)\theta(s)\,d\sigma_s \quad \forall t \in \partial \Omega,\\
& M_2[\epsilon](\theta)(t):= \int_{\partial \Omega}R_{S_n(\cdot,k)}(\epsilon(t-s))\theta(s)\,d\sigma_s \quad \forall t \in \partial \Omega,\\
& M_3[\epsilon](\theta)(t):= \int_{\partial \Omega}T^k_n(\epsilon(t-s))\theta(s)\,d\sigma_s \quad \forall t \in \partial \Omega,
\end{align*}
for all $\theta \in C^{0,\alpha}(\partial \Omega)$ and $\epsilon \in \mathopen]-\epsilon_0,\epsilon_0[$, 
where $R_{S_n(\cdot,k)}$ is defined in Proposition \ref{qnSR}. Then $M_1$, $M_2$, $M_3$ are real analytic and we have 
\begin{equation}\label{hlm:smp:eq:bd1}
\begin{split}
\mathcal{S}^{k}_{q,\eta}&[\partial\Omega_{p,\epsilon},\theta((\cdot-p)/\epsilon)](p+\epsilon t)\\
&=\epsilon M_1[\epsilon](\theta)(t)+\epsilon^{n-1} M_2[\epsilon](\theta)(t)+\epsilon^{n-1}(\log \epsilon) k^{n-2} M_3[\epsilon](\theta)(t) \qquad \forall t \in \partial \Omega\, ,
\end{split}
\end{equation}
for all $\theta \in C^{0,\alpha}(\partial \Omega)$ and $\epsilon \in \mathopen]0,\epsilon_0[$.
\end{prop}
\begin{proof}  Let
\[
M_1^\sharp[\epsilon,\theta](t):= \int_{\partial \Omega}S_n(t-s,\epsilon k)\theta(s)\,d\sigma_s \quad \forall t \in \partial \Omega
\]
for all $(\epsilon,\theta)\in\mathopen]-\epsilon_0,\epsilon_0\mathclose[\times  C^{0,\alpha}(\partial\Omega)$.
By  Lanza de Cristoforis and Rossi \cite[Thm.~4.11]{LaRo08} we deduce that
\[
\mathopen]-\epsilon_0,\epsilon_0\mathclose[\times  C^{0,\alpha}(\partial\Omega)\ni     (\epsilon,\theta)\mapsto M_1^\sharp[\epsilon,\theta]\in C^{1,\alpha}(\partial\Omega)
\]
is real analytic. {Incidentally, we point out that it is exactly in the latter step where we {use} the holomorphic dependence of the fundamental solution upon $k$.}
Since $M_1^\sharp$ is linear and continuous with respect to the variable $\theta$, we have
\[
M_1[\epsilon]=d_\theta M_1^\sharp [\epsilon, \tilde{\theta}]
\qquad\forall (\epsilon, \tilde{\theta})\in \mathopen]-\epsilon_0,\epsilon_0\mathclose[\times  C^{0,\alpha}(\partial\Omega)\,.
\]
where $d_\theta M_1^\sharp [\epsilon, \tilde{\theta}]$ denotes the partial differential with respect to the variable $\theta$ evaluated at the pair $(\epsilon, \tilde{\theta})\in \mathopen]-\epsilon_0,\epsilon_0\mathclose[\times  C^{0,\alpha}(\partial\Omega)$.
Since the right-hand side equals a partial Fr\'{e}chet differential of a  map which is real analytic, the right-hand side is analytic on $(\epsilon, \tilde{\theta})$. Hence $(\epsilon, \tilde{\theta}) \mapsto M_1[\epsilon]$ is real analytic on $\mathopen]-\epsilon_0,\epsilon_0\mathclose[\times  C^{0,\alpha}(\partial\Omega)$ and, since it does not depend on $\tilde{\theta}$, we conclude that it is real analytic on  $\mathopen]-\epsilon_0,\epsilon_0[$. 

A similar argument shows that $M_2$ and $M_3$ are real analytic as well. We just need to replace \cite[Thm.~4.11]{LaRo08} with the analyticity results for the integral operators with real analytic kernel of \cite{LaMu13}. 

Finally, by the definition of $M_1$, $M_2$, and $M_3$, by equality \eqref{hlm:n:eq:Snk},  and by a direct computation based on the theorem of change of variable in integrals, we can verify the validity of equation \eqref{hlm:smp:eq:bd1}.
\end{proof}

By the equality $S_n(\cdot,0)=S_n(\cdot)$ and by standard properties of real analytic maps in Banach spaces, one deduces the validity of the following.

\begin{cor}\label{hlm:smp:cor:bd1}
Let $\alpha \in \mathopen]0,1[$, $q \in \mathbb{D}^+_n(\mathbb{R})$, $\eta \in \mathbb{R}^n$. Let $k \in \mathbb{C}$. Let $\Omega$ be as in assumption \eqref{Omega_def}. Let $p \in Q$. Let $\epsilon_0$ be as in assumption \eqref{epsilon_0}. Let $M_2$, $M_3$ be  as in Proposition \ref{hlm:smp:cor:bd1}. Then there exist $\tilde{\epsilon}'_0 \in \mathopen]0,\epsilon_0[$ and a real analytic map $\tilde{M}_1$  from $\mathopen]-\tilde{\epsilon}'_0,\tilde{\epsilon}'_0[$ to $\mathcal{L}(C^{0,\alpha}(\partial \Omega),C^{1,\alpha}(\partial \Omega))$ such that
\[
\begin{split}
\mathcal{S}^{k}_{q,\eta}&[\partial\Omega_{p,\epsilon},\theta((\cdot-p)/\epsilon)](p+\epsilon t)\\
&=\epsilon  \int_{\partial \Omega}S_n(t-s)\theta(s)\,d\sigma_s +\epsilon^2 \tilde{M}_1[\epsilon](\theta)(t)+\epsilon^{n-1} M_2[\epsilon](\theta)(t)\\
&\quad +\epsilon^{n-1}(\log \epsilon) k^{n-2} M_3[\epsilon](\theta)(t) \qquad \forall t \in \partial \Omega\, ,
\end{split}
\]
for all $\theta \in C^{0,\alpha}(\partial \Omega)$ and $\epsilon \in \mathopen]0,\tilde{\epsilon}'_0[$.
\end{cor}

We now turn to the map from $C^{0,\alpha}(\partial \Omega)$ to $C^{0,\alpha}(\partial \Omega)$ that takes a density $\theta$ to the function  
\[
\left(\mathcal{K}^{k}_{q,\eta}\right)^*[\partial\Omega_{p,\epsilon},\theta((\cdot-p)/\epsilon)](p+\epsilon t)\quad\forall t\in\partial\Omega\, ,
\]
which appears in the formula for the normal derivative of the single layer potential. We study its dependence upon $\epsilon$.
\begin{prop}\label{hlm:smp:prop:bd2}
Let $\alpha \in \mathopen]0,1[$, $q \in \mathbb{D}^+_n(\mathbb{R})$, $\eta \in \mathbb{R}^n$. Let $k \in \mathbb{C}$. Let $\Omega$ be as in assumption \eqref{Omega_def}. Let $p \in Q$. Let $\epsilon_0$ be as in assumption \eqref{epsilon_0}. Let $N_1$, $N_2$, $N_3$ be the maps from $\mathopen]-\epsilon_0,\epsilon_0[$ to $\mathcal{L}(C^{0,\alpha}(\partial \Omega),C^{0,\alpha}(\partial \Omega))$, defined by
\begin{align*}
& N_1[\epsilon](\theta)(t):= \int_{\partial \Omega}\nu_{\Omega}(t)\cdot \nabla S_n(t-s,\epsilon k)\theta(s)\,d\sigma_s \quad \forall t \in \partial \Omega,\\
& N_2[\epsilon](\theta)(t):= \int_{\partial \Omega}\nu_{\Omega}(t)\cdot \nabla R_{S_n(\cdot,k)}(\epsilon(t-s))\theta(s)\,d\sigma_s \quad \forall t \in \partial \Omega,\\
& N_3[\epsilon](\theta)(t):= \int_{\partial \Omega}\nu_{\Omega}(t)\cdot \nabla T^k_n(\epsilon(t-s))\theta(s)\,d\sigma_s \quad \forall t \in \partial \Omega,
\end{align*}
for all $\theta \in C^{0,\alpha}(\partial \Omega)$ and $\epsilon \in \mathopen]-\epsilon_0,\epsilon_0[$. Then $N_1$, $N_2$, $N_3$ are real analytic and we have
\begin{equation}\label{hlm:smp:eq:bd2}
\begin{split}
\left(\mathcal{K}^{k}_{q,\eta}\right)^*&[\partial\Omega_{p,\epsilon},\theta((\cdot-p)/\epsilon)](p+\epsilon t)\\
&= N_1[\epsilon](\theta)(t)+\epsilon^{n-1} N_2[\epsilon](\theta)(t)+\epsilon^{n-1}(\log \epsilon) k^{n-2} N_3[\epsilon](\theta)(t) \qquad \forall t \in \partial \Omega\, ,
\end{split}
\end{equation}
for all $\theta \in C^{0,\alpha}(\partial \Omega)$ and $\epsilon \in \mathopen]0,\epsilon_0[$.\end{prop}
\begin{proof}
By arguing as in the proof of Proposition \ref{hlm:smp:prop:bd1} for $M_1$, $M_2$, $M_3$, we can verify that $N_1$, $N_2$, $N_3$ are real analytic maps from $\mathopen]-\epsilon_0,\epsilon_0[$ to $\mathcal{L}(C^{0,\alpha}(\partial \Omega),C^{0,\alpha}(\partial \Omega))$. Then, by the definition of $N_1$, $N_2$, and $N_3$, by equality \eqref{hlm:n:eq:DSnk}, and by a direct computation based on the theorem of change of variable in integrals, we verify the validity of equation \eqref{hlm:smp:eq:bd2}.
\end{proof}

{As already seen for Corollary \ref{hlm:smp:cor:bd1},  the equality $S_n(\cdot,0)=S_n(\cdot)$ and standard properties of real analytic maps in Banach spaces imply the validity of the following.}

\begin{cor}\label{hlm:smp:cor:bd2}
Let $\alpha \in \mathopen]0,1[$, $q \in \mathbb{D}^+_n(\mathbb{R})$, $\eta \in \mathbb{R}^n$. Let $k \in \mathbb{C}$. Let $\Omega$ be as in assumption \eqref{Omega_def}. Let $p \in Q$. Let $\epsilon_0$ be as in assumption \eqref{epsilon_0}. Let $N_2$, $N_3$ be   as in Proposition \ref{hlm:smp:cor:bd2}. Then there exist $\tilde{\epsilon}''_0 \in \mathopen]0,\epsilon_0[$ and a real analytic map $\tilde{N}_1$  from $]-\tilde{\epsilon}''_0,\tilde{\epsilon}''_0[$ to $\mathcal{L}(C^{0,\alpha}(\partial \Omega),C^{0,\alpha}(\partial \Omega))$ such that
\[
\begin{split}
\left(\mathcal{K}^{k}_{q,\eta}\right)^*&[\partial\Omega_{p,\epsilon},\theta((\cdot-p)/\epsilon)](p+\epsilon t)\\
&= \int_{\partial \Omega}\nu_{\Omega}(t)\cdot \nabla S_n(t-s)\theta(s)\,d\sigma_s+ \epsilon \tilde{N}_1[\epsilon](\theta)(t)+\epsilon^{n-1} N_2[\epsilon](\theta)(t)\\
&\quad+\epsilon^{n-1}(\log \epsilon) k^{n-2} N_3[\epsilon](\theta)(t) \qquad \forall t \in \partial \Omega\, ,
\end{split}
\]
for all $\theta \in C^{0,\alpha}(\partial \Omega)$ and $\epsilon \in \mathopen]0,\tilde{\epsilon}''_0[$.
\end{cor}

{Finally, in} the proposition below, we consider the behavior of the quasi-periodic single layer potential restricted to a set $V$  such that $\overline{V} \cap (p+q\mathbb{Z}^n)=\emptyset$. This somehow characterizes the behavior of ${\mathcal{S}^{k,-}_{q,\eta}}[\partial\Omega_{p,\epsilon},\theta((\cdot-p)/\epsilon)](x)$ when $x$ is far from the {holes}.

\begin{prop}\label{hlm:smp:prop:rep}
Let $\alpha \in \mathopen]0,1[$, $q \in \mathbb{D}^+_n(\mathbb{R})$, $\eta \in \mathbb{R}^n$. Let $k \in \mathbb{C}$. Let $\Omega$ be as in assumption \eqref{Omega_def}. Let $p \in Q$. Let $\epsilon_0$ be as in assumption \eqref{epsilon_0}. 
Let $V$ be a bounded open subset of $\mathbb{R}^n$ such that $\overline{V} \cap (p+q\mathbb{Z}^n)=\emptyset$. Let $\epsilon_{V} \in \mathopen]0,\epsilon_0[$ be such that 
\begin{equation}\label{hlm:smp:eq:rep0}
\overline{V} \subseteq \mathbb{S}_q^-[\Omega_{p,\epsilon}] \qquad \forall \epsilon \in \mathopen]-\epsilon_{V},\epsilon_{V}[.
\end{equation}
Let $M$ be the map from $\mathopen]-\epsilon_{V},\epsilon_{V}[$ to  $\mathcal{L}(C^{0,\alpha}(\partial \Omega),C^2(\overline{V}))$ defined by
\[
M[\epsilon](\theta)(x):= \int_{\partial \Omega}{G^{k}_{q,\eta}}(x-p-\epsilon s)\theta(s)\,d\sigma_s  \qquad \forall x \in \overline{V},
\]
for all $\theta \in C^{0,\alpha}(\partial \Omega)$ and $\epsilon \in \mathopen]-\epsilon_{V},\epsilon_{V}[$. Then $M$ is real analytic and we have
\begin{equation}\label{hlm:smp:eq:rep3}
{\mathcal{S}^{k,-}_{q,\eta}}[\partial\Omega_{p,\epsilon},\theta((\cdot-p)/\epsilon)](x)=\epsilon^{n-1} M[\epsilon](\theta)(x) \qquad \forall x \in \overline{V}\, ,
\end{equation}
for all $\theta \in C^{0,\alpha}(\partial \Omega)$ and $\epsilon \in \mathopen]0,\epsilon_{V}[$.
\end{prop}
\begin{proof}
Since $\epsilon_{V}$ is such that \eqref{hlm:smp:eq:rep0} holds, then we have
\[
\overline{V} -(p+\epsilon \partial \Omega)\subseteq \mathbb{R}^n \setminus q\mathbb{Z}^n \qquad \forall \epsilon \in \mathopen]-\epsilon_{V},\epsilon_{V}[.
\]
By arguing as in the proof of Proposition \ref{hlm:smp:prop:bd1}, one verifies that $M$ is a real analytic {map} from $\mathopen]-\epsilon_{V},\epsilon_{V}[$ to  $\mathcal{L}(C^{0,\alpha}(\partial \Omega),C^2(\overline{V}))$. Then  by the definition of $M$ and by a direct computation based on the theorem of change of variable in integrals, we verify the validity of equation \eqref{hlm:smp:eq:rep3}.
\end{proof}

\subsection{Singular perturbations for the quasi-periodic double layer potential for the Helmholtz equation}

In this subsection, we consider the behavior of the quasi-periodic double layer potential upon singular domain perturbations. As we have done in the previous subsection, we begin by studying the behavior of the quasi-periodic double layer potential restricted to the boundary of {$\Omega_{p,\epsilon}$}. More precisely, we consider the map from $C^{1,\alpha}(\partial \Omega)$ to $C^{1,\alpha}(\partial \Omega)$ {that takes $\theta$ to the function
\[
\mathcal{K}^{k}_{q,\eta}[\partial\Omega_{p,\epsilon},\theta((\cdot-p)/\epsilon)](p+\epsilon t)\, 
\]
of the variable $t \in \partial \Omega$.}

\begin{prop}\label{hlm:dbl:prop:bd2}
Let $\alpha \in \mathopen]0,1[$, $q \in \mathbb{D}^+_n(\mathbb{R})$, $\eta \in \mathbb{R}^n$. Let $k \in \mathbb{C}$. Let $\Omega$ be as in assumption \eqref{Omega_def}. Let $p \in Q$. Let $\epsilon_0$ be as in assumption \eqref{epsilon_0}. Let $P_1$, $P_2$, $P_3$ be the maps from $]-\epsilon_0,\epsilon_0[$ to $\mathcal{L}(C^{1,\alpha}(\partial \Omega),C^{1,\alpha}(\partial \Omega))$, defined by
\begin{align*}
& P_1[\epsilon](\theta)(t):= -\int_{\partial \Omega}\nu_{\Omega}(s)\cdot \nabla S_n(t-s,\epsilon k)\theta(s)\,d\sigma_s \quad \forall t \in \partial \Omega,\\
& P_2[\epsilon](\theta)(t):= -\int_{\partial \Omega}\nu_{\Omega}(s)\cdot \nabla R_{{S_n(\cdot,k)}}(\epsilon(t-s))\theta(s)\,d\sigma_s \quad \forall t \in \partial \Omega,\\
& P_3[\epsilon](\theta)(t):= -\int_{\partial \Omega}\nu_{\Omega}(s)\cdot \nabla {T^k_n}(\epsilon(t-s))\theta(s)\,d\sigma_s \quad \forall t \in \partial \Omega,
\end{align*}
for all $\theta \in C^{1,\alpha}(\partial \Omega)$ and $\epsilon \in \mathopen]-\epsilon_0,\epsilon_0[$. Then $P_1$, $P_2$, $P_3$ are real analytic and we have
\begin{equation}\label{hlm:dbl:eq:bd2}
\begin{split}
\mathcal{K}^{k}_{q,\eta}&[\partial\Omega_{p,\epsilon},\theta((\cdot-p)/\epsilon)](p+\epsilon t)\\
&= P_1[\epsilon](\theta)(t)+\epsilon^{n-1} P_2[\epsilon](\theta)(t)+\epsilon^{n-1}(\log \epsilon) k^{n-2} P_3[\epsilon](\theta)(t) \qquad \forall t \in \partial \Omega\, ,
\end{split}
\end{equation}
for all $\theta \in C^{1,\alpha}(\partial \Omega)$ and $\epsilon \in \mathopen]0,\epsilon_0[$.
\end{prop}
\begin{proof}
By  Lanza de Cristoforis and Rossi \cite[Thm.~4.11]{LaRo08},  we deduce that
\[
\begin{split}
\mathopen]-\epsilon_0,\epsilon_0\mathclose[\times  C^{0,\alpha}(\partial\Omega)\ni     (\epsilon,\theta)\mapsto P_1^\sharp[\epsilon,\theta](t):=  -\int_{\partial \Omega}\nu_{\Omega}(s)\cdot \nabla S_n(t-s,\epsilon k)&\theta(s)\,d\sigma_s \\ 
&\forall t \in \partial \Omega,
\end{split}
\]
is real analytic.
Since $P_1^\sharp$ is linear and continuous with respect to the variable $\theta$, we have
\[
P_1[\epsilon]=d_\theta P_1^\sharp [\epsilon, \tilde{\theta}]
\qquad\forall (\epsilon, \tilde{\theta})\in \mathopen]-\epsilon_0,\epsilon_0\mathclose[\times  C^{1,\alpha}(\partial\Omega)\,.
\]
where $d_\theta P_1^\sharp [\epsilon, \tilde{\theta}]$ denotes the partial differential with respect to the variable $\theta$ evaluated at the pair $(\epsilon, \tilde{\theta})\in \mathopen]-\epsilon_0,\epsilon_0\mathclose[\times  C^{1,\alpha}(\partial\Omega)$.
Since the right-hand side equals a partial Fr\'{e}chet differential of a  map which is real analytic, the right-hand side is analytic on $(\epsilon, \tilde{\theta})$. Hence $(\epsilon, \tilde{\theta}) \mapsto P_1[\epsilon]$ is real analytic on $\mathopen]-\epsilon_0,\epsilon_0\mathclose[\times  C^{1,\alpha}(\partial\Omega)$ and, since it does not depend on $\tilde{\theta}$, we conclude that it is real analytic on  $\mathopen]-\epsilon_0,\epsilon_0[$.  By exploiting  the regularity results for the integral operators with real analytic kernel of \cite{LaMu13} and by arguing so as to prove the real analyticity of $P_1$, we can prove that $P_2$ and $P_3$ are real analytic.  Finally, by the definition of $P_1$, $P_2$, and $P_3$, by equality \eqref{hlm:n:eq:DSnk}, and by a direct computation based on the theorem of change of variable in integrals, we verify the validity of equation \eqref{hlm:dbl:eq:bd2}.
\end{proof}

{As is the previous subsection, one can immediately deduce the following Corollary \ref{hlm:dbl:cor:bd2}.} 

\begin{cor}\label{hlm:dbl:cor:bd2}
Let $\alpha \in \mathopen]0,1[$, $q \in \mathbb{D}^+_n(\mathbb{R})$, $\eta \in \mathbb{R}^n$. Let $k \in \mathbb{C}$. Let $\Omega$ be as in assumption \eqref{Omega_def}. Let $p \in Q$. Let $\epsilon_0$ be as in assumption \eqref{epsilon_0}. Let $P_2$, $P_3$ be   as in Proposition \ref{hlm:dbl:cor:bd2}. Then there exist $\tilde{\epsilon}'''_0 \in \mathopen]0,\epsilon_0[$ and a real analytic map $\tilde{P}_1$  from $\mathopen]-\tilde{\epsilon}'''_0,\tilde{\epsilon}'''_0[$ to $\mathcal{L}(C^{1,\alpha}(\partial \Omega),C^{1,\alpha}(\partial \Omega))$ such that
\[
\begin{split}
\mathcal{K}^{k}_{q,\eta}&[\partial\Omega_{p,\epsilon},\theta((\cdot-p)/\epsilon)](p+\epsilon t)\\
&= -\int_{\partial \Omega}\nu_{\Omega}(s)\cdot \nabla S_n(t-s)\theta(s)\,d\sigma_s+\epsilon \tilde{P}_1[\epsilon](\theta)(t)+\epsilon^{n-1} P_2[\epsilon](\theta)(t)\\
&\quad +\epsilon^{n-1}(\log \epsilon) k^{n-2} P_3[\epsilon](\theta)(t) \qquad \forall t \in \partial \Omega\, ,
\end{split}
\]
for all $\theta \in C^{1,\alpha}(\partial \Omega)$ and $\epsilon \in \mathopen]0,\tilde{\epsilon}'''_0[$.
\end{cor}

In the proposition below, we consider the behavior of the quasi-periodic double layer potential far from the {holes}.

\begin{prop}\label{hlm:dbl:prop:rep}
Let $\alpha \in \mathopen]0,1[$, $q \in \mathbb{D}^+_n(\mathbb{R})$, $\eta \in \mathbb{R}^n$. Let $k \in \mathbb{C}$. Let $\Omega$ be as in assumption \eqref{Omega_def}. Let $p \in Q$. Let $\epsilon_0$ be as in assumption \eqref{epsilon_0}. 
Let $V$ be a bounded open subset of $\mathbb{R}^n$ such that $\overline{V} \cap (p+q\mathbb{Z}^n)=\emptyset$. Let $\epsilon_{V} \in \mathopen]0,\epsilon_0[$ be such that 
\begin{equation}\label{hlm:dbl:eq:rep0}
\overline{V} \subseteq \mathbb{S}_q^-[\Omega_{p,\epsilon}] \qquad \forall \epsilon \in \mathopen]-\epsilon_{V},\epsilon_{V}[.
\end{equation}
Let $P$ be the map from $\mathopen]-\epsilon_{V},\epsilon_{V}[$ to  $\mathcal{L}(C^{1,\alpha}(\partial \Omega),C^2(\overline{V}))$ defined by
\[
P[\epsilon](\theta)(x):= - \int_{\partial \Omega}\nu_\Omega(s)\cdot \nabla {G^{k}_{q,\eta}}(x-p-\epsilon s)\theta(s)\,d\sigma_s  \qquad \forall x \in \overline{V},
\]
for all $\theta \in C^{1,\alpha}(\partial \Omega)$ and $\epsilon \in \mathopen]-\epsilon_{V},\epsilon_{V}[$. Then $P$ is real analytic and we have
\begin{equation}\label{hlm:dbl:eq:rep2}
\mathcal{D}^{k,-}_{q,\eta}[\partial\Omega_{p,\epsilon},\theta((\cdot-p)/\epsilon)](x)=\epsilon^{n-1} P[\epsilon](\theta)(x) \qquad \forall x \in \overline{V}\, ,
\end{equation}
for all $\theta \in C^{1,\alpha}(\partial \Omega)$ and $\epsilon \in \mathopen]0,\epsilon_{V}[$.
\end{prop}
\begin{proof}
Since $\epsilon_{V}$ is such that \eqref{hlm:dbl:eq:rep0} holds, then we have
\[
\overline{V} -(p+\epsilon \partial \Omega)\subseteq \mathbb{R}^n \setminus q\mathbb{Z}^n \qquad \forall \epsilon \in \mathopen]-\epsilon_{V},\epsilon_{V}[.
\]
By arguing as in the proof of Proposition \ref{hlm:dbl:prop:bd2}, one verifies that $P$ is a real analytic map from $\mathopen]-\epsilon_{V},\epsilon_{V}[$ to  $\mathcal{L}(C^{1,\alpha}(\partial \Omega),C^2(\overline{V}))$. Then  by the definition of $P$ and by a direct computation based on the theorem of change of variable in integrals, we verify the validity of equation \eqref{hlm:dbl:eq:rep2}.
\end{proof}

\section{A singularly perturbed nonlinear quasi-periodic boundary value problem for the Helmholtz equation}\label{s:singbvp}

\begin{figure}
\begin{center}
\includegraphics[width=3.7in]{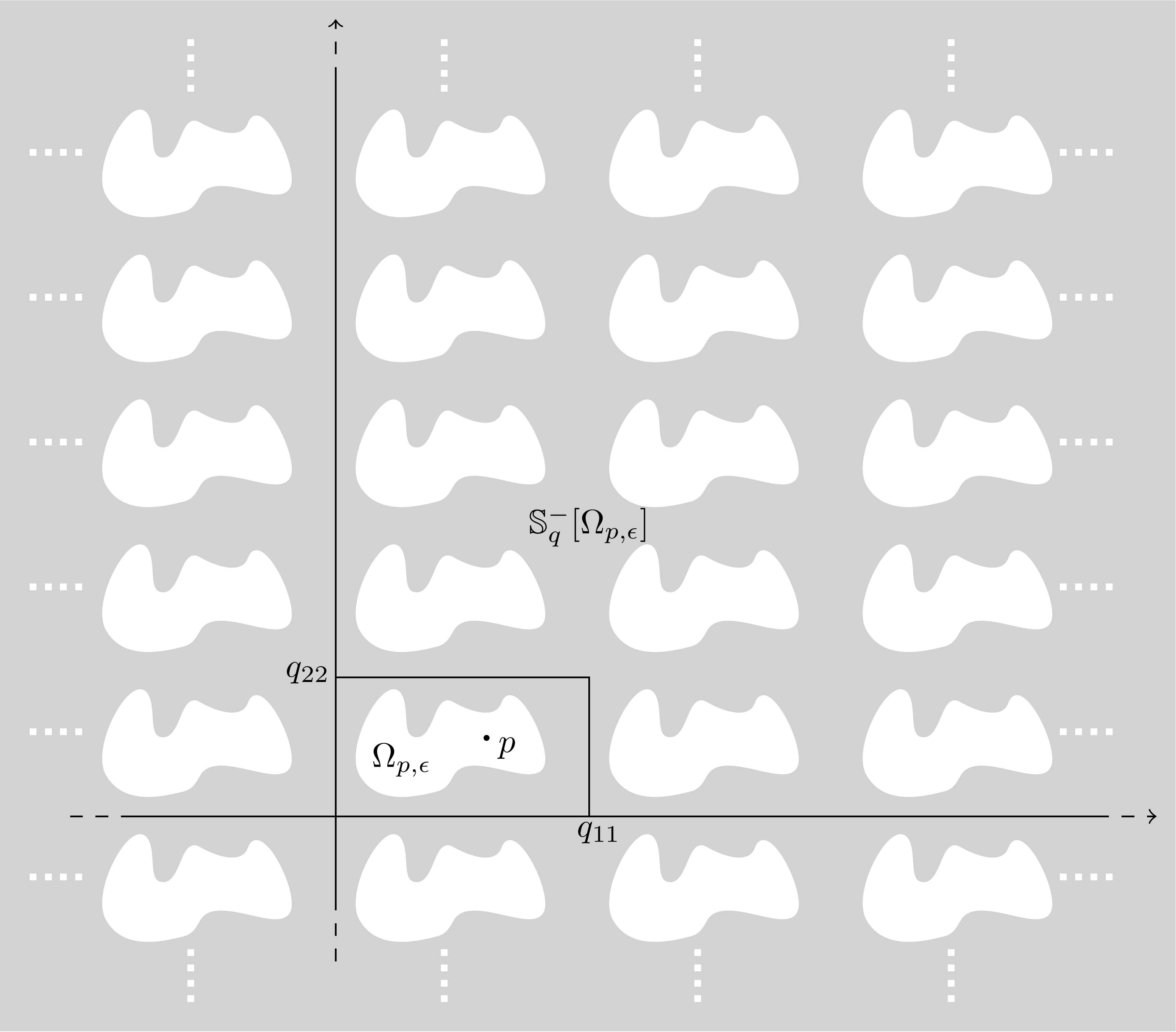}
\caption{{ An example, in dimension $n=2$, of the geometric setting. The grey region is the periodically perforeted set
${\mathbb{S}_q^-[\Omega_{p,\epsilon}]}$. }}
\label{fig:geom-sett}
\end{center}
\end{figure}

In this section, we study the asymptotic behavior of a nonlinear quasi-periodic Robin boundary value problem for the Helmholtz equation in the singularly perturbed domain $\mathbb{S}^-_q[\Omega_{p,\epsilon}]$ as $\epsilon \to 0^+$ (see Figure 
\ref{fig:geom-sett}). We choose this specific problem  to show that our method can be applied to a nonlinear boundary value problem for the quasi-periodic Helmholtz equation. The nonlinearity comes from the boundary condition, which prescribes that the  normal derivative of the solution equals a nonlinear function of its trace on the boundary. The strategy that we are going to use in our analysis was successfully employed by Lanza de Cristoforis in \cite{La07} to study a nonlinear Robin problem for the Laplace equation in a bounded domain with a small hole, and then later in \cite{LaMu13bis} for the case of a periodically perforated domain. An application to nonlinear problems for the Lam\'e equations can be found, for example, in \cite{DaLa11}.

Let $\epsilon_0$ be as in assumption \eqref{epsilon_0}.  {Let {$k \in \mathbb{C}$ be such} that $k^2 \notin \sigma_{q,\eta}(-\Delta)$.}
Possibly taking   $\epsilon^\#_0 \in \mathopen]0,\epsilon_0[$ small enough, we {can} assume that 
\begin{equation}\label{assepsbis}
k^2 \notin  \sigma^N_{q,\eta}(-\Delta, \mathbb{S}^-_q[\Omega_{p,\epsilon}])\, , \qquad k^2 \notin \sigma^{D}(-\Delta,\Omega_{p,\epsilon}) \qquad \forall \epsilon \in \mathopen]0,\epsilon^\#_0[\, {.}
\end{equation}
{{Condition} $k^2 \notin \sigma^{D}(-\Delta,\Omega_{p,\epsilon})$ follows by a simple rescaling argument {since  
\[
\sigma^{D}(-\Delta,\Omega_{p,\epsilon})=\epsilon^{-2} \sigma^{D}(-\Delta,\Omega).
\]}
{Concerning} condition $k^2 \notin  \sigma^N_{q,\eta}(-\Delta, \mathbb{S}^-_q[\Omega_{p,\epsilon}])$, {
 the }} argument of Rauch and Taylor \cite{RaTa75}, which can be extended to $(Q,\eta)$-quasi-periodic eigenvalues by arguing exactly as it is done in \cite[Chap.~7]{Mu12bis} for the {periodic case,
 shows that Neumann $(Q,\eta)$-quasi-periodic eigenvalues on $\mathbb{S}_q^-[\Omega_{p,\epsilon}]$ converge to  $(Q,\eta)$-quasi-periodic eigenvalues on $\mathbb{R}^n$ as $\epsilon \to 0^+$}. {Of course we can also assume that} 
\begin{equation}\label{assepster}
\epsilon^\#_0 < \min\{\tilde{\epsilon}'_0,\tilde{\epsilon}''_0\}\, ,
\end{equation}
where $\tilde{\epsilon}'_0$ and $\tilde{\epsilon}''_0$ are as in Corollaries \ref{hlm:smp:cor:bd1} and \ref{hlm:smp:cor:bd2}, respectively. Assumption \eqref{assepsbis} will allow to use the representation formula of Corollary \ref{cor:repneu} for quasi-periodic solutions of the Helmholtz equation, whereas \eqref{assepster} will be used in order to exploit Corollaries \ref{hlm:smp:cor:bd1} and \ref{hlm:smp:cor:bd2} on some boundary integral operators. Let
\[
\mathcal{B}\colon C^{0,\alpha}(\partial\Omega)\to C^{0,\alpha}(\partial\Omega)
\]
{be} a generic map which we will assume to have a certain degree of regularity when needed.
Then for $\epsilon\in\mathopen]0,\epsilon^\#_{0}[$ we consider the following nonlinear problem
\begin{equation}\label{bvpe}
\begin{cases}
\Delta u + k^2 u=0 \qquad &\mbox{ in } \mathbb{S}^-_q[\Omega_{p,\epsilon}],\\
u \, \mbox{is $(Q,\eta)$-quasi-periodic},\\
\partial_{\nu_{\Omega_{p,\epsilon}}} u(x)=\mathcal{B}[u(p+\epsilon \cdot)]((x-p)/\epsilon) \qquad &\forall x \in  \partial\Omega_{p,\epsilon}.
\end{cases}
\end{equation}

{\it A priori}, we do not know whether problem \eqref{bvpe} has solutions or not. We will show that, for $\epsilon$ small enough, problem \eqref{bvpe} has at least one solution that we will denote by $u(\epsilon,\cdot)$. After that, our aim will be to study the asymptotic behavior of $u(\epsilon,\cdot)$ as $\epsilon \to 0^+$ in terms of real analytic maps and of possibly singular but explicitly known functions of the singular perturbation parameter $\epsilon$.  We observe that one could also prove a local uniqueness property for the family of solutions $u(\epsilon,\cdot)$ by following the arguments of \cite{DaMoMu20}.

\subsection{An integral equation formulation of problem (\ref{bvpe})}

We plan to transform problem \eqref{bvpe} into an equivalent integral equation defined on a set that does not depend on $\epsilon$. In a sense, the dependence on $\epsilon$ will  pass from the geometry of the problem (the set $\Omega_{p,\epsilon}$) to the operators in the integral equation. The first step consists in combining the representation formula of Corollary \ref{cor:repneu} and a simple rescaling argument in order to establish a bijection between the solutions of problem \eqref{bvpe} and the solutions of a nonlinear boundary integral equation on $\partial \Omega$.

\begin{prop}\label{prop:bij}
Let $\alpha \in \mathopen]0,1[$, $q \in \mathbb{D}^+_n(\mathbb{R})$, $\eta \in \mathbb{R}^n$. Let {$k \in \mathbb{C}$ be such} that $k^2 \notin \sigma_{q,\eta}(-\Delta)$. Let $\Omega$ be as in assumption \eqref{Omega_def}. Let $p \in Q$. Let $\epsilon^\#_0$ be as in assumptions \eqref{assepsbis} and \eqref{assepster}.  Let $\epsilon \in \mathopen]0,\epsilon^\#_0[$. Let $\mathcal{B}$ be a map from $C^{0,\alpha}(\partial\Omega)$ to $C^{0,\alpha}(\partial\Omega)$. 
Then the map $u_{\epsilon}[\cdot]$ from the set of $\theta\in 
C^{0,\alpha}(\partial\Omega)$ which satisfy the equation
\begin{equation}\label{eq:bij}
\begin{split}
\frac{1}{2}\theta(t)+\left(\mathcal{K}^{k}_{q,\eta}\right)^*&[\partial\Omega_{p,\epsilon},\theta((\cdot-p)/\epsilon)](p+\epsilon t)\\
&={\mathcal{B}[\mathcal{S}^{k}_{q,\eta}[\partial\Omega_{p,\epsilon},\theta((\cdot-p)/\epsilon)](p+\epsilon \cdot)](t)} \qquad \forall t \in \partial \Omega\, ,
\end{split}
\end{equation}
 to the set of solutions $u\in C^{1,\alpha}(\overline{{\mathbb{S}}_q^-[\Omega_{p,\epsilon}]})$ of problem \eqref{bvpe}, that takes $\theta$ to the function
 \[
u_\epsilon[\theta](x):= {\mathcal{S}^{k,-}_{q,\eta}}[\partial\Omega_{p,\epsilon},\theta((\cdot-p)/\epsilon)](x) \qquad \forall x \in \overline{{\mathbb{S}}_q^-[\Omega_{p,\epsilon}]}\, ,
 \]
 is a bijection. 
\end{prop}
\begin{proof}
We first note that if $u\in C^{1,\alpha}(\overline{{\mathbb{S}}_q^-[\Omega_{p,\epsilon}]})$ is a solution of problem \eqref{bvpe}, then by Corollary \ref{cor:repneu} there exists a unique $\theta \in C^{0,\alpha}(\partial \Omega)$ such that
 \[
u(x)= {\mathcal{S}^{k,-}_{q,\eta}}[\partial\Omega_{p,\epsilon},\theta((\cdot-p)/\epsilon)](x) \qquad \forall x \in \overline{{\mathbb{S}}_q^-[\Omega_{p,\epsilon}]}\, .
 \]
 Then by Proposition \ref{prop:slp} one immediately verifies that $\theta$ must solve equation \eqref{eq:bij}. Conversely, if  $\theta \in C^{0,\alpha}(\partial \Omega)$  solves equation \eqref{eq:bij}, by Proposition \ref{prop:slp} we deduce that ${\mathcal{S}^{k,-}_{q,\eta}}[\partial\Omega_{p,\epsilon},\theta((\cdot-p)/\epsilon)]$ is in $C^{1,\alpha}(\overline{{\mathbb{S}}_q^-[\Omega_{p,\epsilon}]})$ and solves  problem \eqref{bvpe}.
\end{proof}

Our next goal is to analyze the solutions of equation \eqref{eq:bij}. More specifically, we are interested in understanding their dependence on $\epsilon$. We now set 
\[
\mathcal{V}[\partial\Omega,\theta](t):= \int_{\partial \Omega} S_n(t-s)\theta(s)\,d\sigma_s \qquad \forall t \in \partial \Omega\, ,
\]
and
\[
\mathcal{K}^*[\partial\Omega,\theta](t):= \int_{\partial \Omega}\nu_{\Omega}(t)\cdot \nabla S_n(t-s)\theta(s)\,d\sigma_s \qquad \forall t \in \partial \Omega\, ,
\]
for all $\theta \in C^{0,\alpha}(\partial \Omega)$. By classical potential theory (see \cite[Thms.~4.25 and 6.7]{DaLaMu21}), we know that $\mathcal{V}[\partial\Omega,\cdot]$ is a bounded linear operator from $C^{0,\alpha}(\partial \Omega)$ to $C^{1,\alpha}(\partial \Omega)$ and that $\mathcal{K}^*[\partial\Omega,\cdot]$ is a {compact} linear operator from $C^{0,\alpha}(\partial \Omega)$ to $C^{0,\alpha}(\partial \Omega)$. Then we observe that if $\epsilon \in \mathopen]0,\epsilon^\#_0[$ {by Corollaries \ref{hlm:smp:cor:bd1} and \ref{hlm:smp:cor:bd2}} equation \eqref{eq:bij} can be rewritten as
\begin{equation}\label{eq:bijbis}
\begin{split}
&\frac{1}{2}\theta+\mathcal{K}^*[\partial\Omega,\theta]+ \epsilon \tilde{N}_1[\epsilon](\theta)+\epsilon^{n-1} N_2[\epsilon](\theta)\\
&\quad+\epsilon^{n-1}(\log \epsilon) k^{n-2} N_3[\epsilon](\theta) -\mathcal{B}\Bigg[\epsilon \mathcal{V}[\partial\Omega,\theta] +\epsilon^2  \tilde{M}_1[\epsilon](\theta)+\epsilon^{n-1}M_2[\epsilon](\theta)\\
&\quad +\epsilon^{n-1}(\log \epsilon) k^{n-2} M_3[\epsilon](\theta)\Bigg]=0 \qquad \mathrm{on}\ \partial \Omega\, .
\end{split}
\end{equation}
We now note that equation \eqref{eq:bijbis} can be seen as an equation which depends on the quantities $\epsilon$ and  $\epsilon (\log \epsilon)$. Therefore, we replace the quantity $\epsilon (\log \epsilon)$ by an auxiliary variable $r$ and, motivated by \eqref{eq:bijbis}, we introduce the map $\Lambda$ from $$\mathopen]-\epsilon^\#_0,\epsilon^\#_0\mathclose[\times \mathbb{R} \times C^{0,\alpha}(\partial \Omega)$$ to $C^{0,\alpha}(\partial \Omega)$ defined by
\[
\begin{split}
\Lambda[\epsilon,r,\theta]:=&\frac{1}{2}\theta+\mathcal{K}^*[\partial\Omega,\theta]+ \epsilon \tilde{N}_1[\epsilon](\theta)+\epsilon^{n-1} N_2[\epsilon](\theta)\\
&\quad+\epsilon^{n-2}r k^{n-2} N_3[\epsilon](\theta) -\mathcal{B}\Bigg[\epsilon  \mathcal{V}[\partial\Omega,\theta] +\epsilon^2 \tilde{M}_1[\epsilon](\theta)+\epsilon^{n-1}M_2[\epsilon](\theta)\\
&\quad + \epsilon^{n-2} rk^{n-2} M_3[\epsilon](\theta)\Bigg] \qquad \mathrm{on}\ \partial \Omega\, ,
\end{split}
\]
for all $(\epsilon,r,\theta)\in \mathopen]-\epsilon^\#_0,\epsilon^\#_0\mathclose[\times \mathbb{R} \times C^{0,\alpha}(\partial \Omega)$. Next we recall that if $n$ is odd, then ${T^k_n}(x)=0$ for all $x \in \mathbb{R}^n$, and thus $N_3$ and $M_3$ are identically equal to $0$. Therefore, we find convenient to introduce the function $n \mapsto \tau_n$ from $\mathbb{N}$ to itself defined by
\[
\tau_n :=
\left\{
\begin{array}{ll}
1  &   \text{if $n$ is even}\, ,\\
0  &   \text{if $n$ is odd}\, .
\end{array}
\right.
\]
Then if $\epsilon \in \mathopen]0,\epsilon^\#_0[$ equation \eqref{eq:bijbis} can be written as
\begin{equation}\label{eq:Lambdaeps}
\Lambda[\epsilon,\tau_n \epsilon (\log \epsilon),\theta]=0\, .
\end{equation}
Our aim is to understand the behavior of the solutions $\theta$ of equation \eqref{eq:Lambdaeps} as $\epsilon$ approaches the degenerate value $\epsilon=0$. Therefore we note that, if we further assume that $\mathcal{B}$ is continuous  from $C^{0,\alpha}(\partial\Omega)$ to $C^{0,\alpha}(\partial\Omega)$, letting $\epsilon \to 0^+$ in equation \eqref{eq:Lambdaeps}, we obtain
\begin{equation}\label{eq:lim}
\Lambda[0,0,\theta]=0\, ,
\end{equation}
or equivalently
\[
\begin{split}
&\frac{1}{2}\theta+\mathcal{K}^*[\partial\Omega,\theta] = {\mathcal{B}}[0] \qquad \mathrm{on}\ \partial \Omega\, .
\end{split}
\]
By \cite[Cor.~6.15]{DaLaMu21}, equation \eqref{eq:lim} in the unknown $\theta$ has a unique solution in $C^{0,\alpha}(\partial \Omega)$, which we denote by $\tilde{\theta}$.

In view of the bijection result of Proposition \ref{prop:bij}, a crucial step to understand the solvability of problem \eqref{bvpe} is to understand the solvability of the integral equation \eqref{eq:Lambdaeps} and the behavior of the solutions. This is done by the Implicit Function Theorem in the proposition  below. Before stating it, we observe that we add the assumption that ${\mathcal{B}}$ is a real analytic map from $C^{0,\alpha}(\partial\Omega)$ to $C^{0,\alpha}(\partial\Omega)$. This is done in order to deduce the real analyticity of the implicitly defined function.

\begin{prop}\label{prop:Theta}
Let $\alpha \in \mathopen]0,1[$, $q \in \mathbb{D}^+_n(\mathbb{R})$, $\eta \in \mathbb{R}^n$. Let {$k \in \mathbb{C}$ be such} that $k^2 \notin \sigma_{q,\eta}(-\Delta)$. Let $\Omega$ be as in assumption \eqref{Omega_def}. Let $p \in Q$. Let $\epsilon^\#_0$ be as in assumptions \eqref{assepsbis} and \eqref{assepster}.  Let ${\mathcal{B}}$ be a real analytic map from $C^{0,\alpha}(\partial\Omega)$ to $C^{0,\alpha}(\partial\Omega)$. Let $\tilde{\theta}$ be the unique solution in $C^{0,\alpha}(\partial \Omega)$ of equation \eqref{eq:lim}. Then there exist $\epsilon' \in \mathopen]0,\epsilon^\#_0[$, $r' \in \mathopen]0,+\infty[$, an open neighborhood $\mathcal{O}_{\tilde{\theta}}$ of $\tilde{\theta}$ in $C^{0,\alpha}(\partial \Omega)$, and a real analytic  map $\Theta$ from $\mathopen]-\epsilon',\epsilon'\mathclose[\times \mathopen]-r',r'[$ to $\mathcal{O}_{\tilde{\theta}}$ such that
\[
\epsilon (\log \epsilon)\in \mathopen]-r',r'[ \qquad \forall \epsilon \in \mathopen]0,\epsilon'[\, ,
\]
and such that the set of zeros of $\Lambda$ in $\mathopen]-\epsilon',\epsilon'\mathclose[\times  \mathopen]-r',r'\mathclose[\times \mathcal{O}_{\tilde{\theta}}$ coincides with the graph of $\Theta$. In particular,  $\Theta[0,0]=\tilde{\theta}$.
\end{prop}
\begin{proof}
We plan to apply  the Implicit Function Theorem for real analytic maps in Banach spaces. By Corollaries \ref{hlm:smp:cor:bd1} and \ref{hlm:smp:cor:bd2}, by classical potential theory (see \cite[Thms. 4.25 and 6.7]{DaLaMu21}), and by standard calculus in Banach spaces, we deduce that $\Lambda$ is a real analytic map from $\mathopen]-\epsilon^\#_0,\epsilon^\#_0\mathclose[\times \mathbb{R} \times C^{0,\alpha}(\partial \Omega)$ to $C^{0,\alpha}(\partial \Omega)$. By standard calculus in Banach spaces, we have that the partial differential $d_\theta \Lambda [0,0,\tilde{\theta}]$ of $\Lambda$ at $(0,0,\tilde{\theta})$ with respect to $\theta$ is delivered by the linear map
\[
d_\theta \Lambda [0,0,\tilde{\theta}]=\frac{1}{2}\mathbb{I}+\mathcal{K}^*[\partial\Omega,\cdot]\, .
\]
As a consequence, by \cite[Cor.~6.15]{DaLaMu21}, $d_\theta \Lambda [0,0,\tilde{\theta}]$ is a linear homeomorphism from $C^{0,\alpha}(\partial \Omega)$ to itself. Then by the Implicit Function Theorem for real analytic maps in Banach spaces (cf.~Deimling \cite[Thm.~15.3]{De85}), we deduce the existence of  $\epsilon' \in \mathopen]0,\epsilon^\#_0[$, $r' \in \mathopen]0,+\infty[$, an open neighborhood $\mathcal{O}_{\tilde{\theta}}$ of $\tilde{\theta}$ in $C^{0,\alpha}(\partial \Omega)$, and a real analytic   $\Theta$ from $\mathopen]-\epsilon',\epsilon'\mathclose[\times \mathopen]-r',r'[$ to $\mathcal{O}_{\tilde{\theta}}$ such that
\[
\epsilon (\log \epsilon)\in \mathopen]-r',r'[ \qquad \forall \epsilon \in \mathopen]0,\epsilon'[\, ,
\]
{the} set of zeros of $\Lambda$ in $\mathopen]-\epsilon',\epsilon'\mathclose[\times  \mathopen]-r',r'\mathclose[\times \mathcal{O}_{\tilde{\theta}}$ coincides with the graph of $\Theta$, {and}   $\Theta[0,0]=\tilde{\theta}$.
\end{proof}

\begin{rem}\label{rem:an}
An example of  a real analytic map ${\mathcal{B}}$  from $C^{0,\alpha}(\partial\Omega)$ to $C^{0,\alpha}(\partial\Omega)$ can be obtained by considering composition operators of the type
\[
u\mapsto {\mathcal{B}}[u](\cdot):=G(u(\cdot))\, ,
\]
where $G$ is a real analytic function from $\mathbb{C}$ to itself (see, e.g., Valent~\cite[Thm.~5.2, p.~44]{Va88}). 
\end{rem}

By means of the densities $\Theta[\epsilon,\tau_n \epsilon (\log \epsilon)]$ for $\epsilon \in \mathopen]0,\epsilon'[$, we can finally introduce in the corollary below the family of solutions $\{u(\epsilon,\cdot)\}_{\epsilon \in \mathopen]0,\epsilon'[}$ to problem \eqref{bvpe}.

\begin{cor}
Let the assumptions of Proposition \ref{prop:Theta} hold.  {We} set
\[
u(\epsilon,x):=u_\epsilon \big[\Theta[\epsilon,\tau_n \epsilon (\log \epsilon)]\big](x)\qquad \forall x \in \overline{{\mathbb{S}}_q^-[\Omega_{p,\epsilon}]}\, ,
\]
for all $\epsilon \in \mathopen]0,\epsilon'[$. Then for each $\epsilon \in \mathopen]0,\epsilon'[$ the function $u(\epsilon,\cdot)$ is a solution of problem \eqref{bvpe}.
\end{cor}

\subsection{A representation formula for the family of solutions $\{u(\epsilon,\cdot)\}_{\epsilon \in ]0,\epsilon'[}$}

We are now ready to prove the main theorem of the section. More precisely, we represent the restrictions to a fixed  subset $V$ such that $\overline{V} \cap (p+q\mathbb{Z}^n)=\emptyset$ of the solutions $\{u(\epsilon,\cdot)\}_{\epsilon \in \mathopen]0,\epsilon'[}$ in terms of a real analytic operator of two variables with values in a function space evaluated at the pair $(\epsilon,\tau_n \epsilon (\log \epsilon))$. This implies the possibility to represent $u(\epsilon,\cdot)$ as a converging power series of the pair $(\epsilon,\tau_n \epsilon (\log \epsilon))$. Similar results can be also obtained for the restrictions of $\{u(\epsilon,\cdot)\}_{\epsilon \in \mathopen]0,\epsilon'[}$ to sets of the type $p+\epsilon \tilde{V}$ or for suitable functionals of $u(\epsilon,\cdot)$, like for example its energy integral. However, for the sake of brevity, here we consider only the previously described result on the restriction of the solution to a fixed set far from the holes.

\begin{thm}\label{thm:rep}
Let $\alpha \in \mathopen]0,1[$, $q \in \mathbb{D}^+_n(\mathbb{R})$, $\eta \in \mathbb{R}^n$. Let {$k \in \mathbb{C}$ be such} that $k^2 \notin \sigma_{q,\eta}(-\Delta)$. Let $\Omega$ be as in assumption \eqref{Omega_def}. Let $p \in Q$. Let $\epsilon^\#_0$ be as in assumptions \eqref{assepsbis} and \eqref{assepster}.   Let ${\mathcal{B}}$ be a real analytic map from $C^{0,\alpha}(\partial\Omega)$ to $C^{0,\alpha}(\partial\Omega)$.  Let $\tilde{\theta}$ be the unique solution in $C^{0,\alpha}(\partial \Omega)$ of equation \eqref{eq:lim}. Let $\epsilon' \in \mathopen]0,\epsilon^\#_0[$, $r' \in \mathopen]0,+\infty[$ be as in Proposition \ref{prop:Theta}. Let $V$ be a bounded open subset of $\mathbb{R}^n$ such that $\overline{V} \cap (p+q\mathbb{Z}^n)=\emptyset$. Then there exist $\epsilon''  \in \mathopen]0,\epsilon']$,  such that 
\begin{equation}\label{eq:rep1}
\overline{V} \subseteq \mathbb{S}_q^-[\Omega_{p,\epsilon}] \qquad \forall \epsilon \in \mathopen]-\epsilon'',\epsilon''[
\end{equation}
and a real analytic map $U$ from  $\mathopen]-\epsilon'',\epsilon''\mathclose[\times  \mathopen]-r',r'[$ to $C^2(\overline{V})$ such that
\begin{equation}\label{eq:rep2}
u(\epsilon,x)=\epsilon^{n-1}U[\epsilon,\tau_n \epsilon (\log \epsilon)](x) \qquad \forall x \in \overline{V}\, 
\end{equation}
for all $\epsilon \in \mathopen]0,\epsilon''[$ and that
\begin{equation}\label{eq:rep3}
U[0,0]= {G^{k}_{q,\eta}}(x-p) \int_{\partial \Omega}\tilde{\theta}\,d\sigma \qquad \forall x \in \overline{V}\, .
\end{equation}
\end{thm}
\begin{proof}
Choosing $\epsilon''$ small enough, we can clearly assume that \eqref{eq:rep1} holds. Then we have
\[
\overline{V} -(p+\epsilon \partial \Omega)\subseteq \mathbb{R}^n \setminus q\mathbb{Z}^n \qquad \forall \epsilon \in \mathopen]-\epsilon'',\epsilon''[.
\]
By the theorem of change of variable in integrals, we have
\[
u(\epsilon,x)= \epsilon^{n-1}\int_{\partial \Omega}{G^{k}_{q,\eta}}(x-p-\epsilon s)\Theta[\epsilon,\tau_n \epsilon (\log \epsilon)](s)\,d\sigma_s  \qquad \forall x \in \overline{V}\ ,
\]
for all $\epsilon \in \mathopen]0,\epsilon''[$. Then we find natural to set
\[
U[\epsilon,r](x)=\int_{\partial \Omega}{G^{k}_{q,\eta}}(x-p-\epsilon s)\Theta[\epsilon,r](s)\,d\sigma_s  \qquad \forall x \in \overline{V}
\]
for all $(\epsilon,r)\in \mathopen]-\epsilon'',\epsilon''\mathclose[\times  \mathopen]-r',r'[$. By the regularity results for the integral operators with real analytic kernel of \cite{LaMu13}, we deduce that  $U$ is a real analytic map from  $\mathopen]-\epsilon'',\epsilon''\mathclose[\times  \mathopen]-r',r'[$ to $C^2(\overline{V})$. By the definition of $U$, we clearly have that equality \eqref{eq:rep2} holds for all $\epsilon \in \mathopen]0,\epsilon''[$.  Since $\Theta[0,0]=\tilde{\theta}$, we also deduce the validity of \eqref{eq:rep3}.
\end{proof}


\section*{Acknowledgements}
 
 The authors are members of the ``Gruppo Nazionale per l'Analisi Matematica, la Probabilit\`a e le loro Applicazioni'' (GNAMPA) of the ``Istituto Nazionale di Alta Matematica'' (INdAM). R.B.~is supported by the FWO Odysseus 1 grant G.0H94.18N: Analysis and Partial Differential Equations and by the Methusalem programme of the Ghent University Special Research Fund (BOF) (Grant number 01M01021). P.M.~acknowledges the support from EU through the H2020-MSCA-RISE-2020 project EffectFact, 
Grant agreement ID: 101008140.


\end{document}